\documentclass[oneside]{amsart}

\usepackage[numbers, sort, longnamesfirst]{natbib}
\usepackage[foot]{amsaddr}
\usepackage{amssymb,amsmath,mathtools}
\usepackage{enumerate}
\usepackage{xcolor}
\usepackage{mathrsfs}
\usepackage{enumitem}
\usepackage{graphicx, subcaption}
\usepackage{float}
\usepackage{import}
\usepackage{xifthen}
\usepackage{pdfpages}
\usepackage{transparent}

\usepackage{setspace}

\usepackage{geometry}
\usepackage{hyperref}
\hypersetup{colorlinks = true, linkcolor = blue, citecolor = blue, pdfborder={0 0 0}, pdfpagemode=UseOutlines, bookmarks = true}
\usepackage{cleveref}
\usepackage{autonum}
\usepackage{dsfont}

\newtheorem{theorem}{Theorem}[section]
\newtheorem{proposition}[theorem]{Proposition}
\newtheorem{lemma}[theorem]{Lemma}
\newtheorem{corollary}[theorem]{Corollary}
\theoremstyle{definition}
\newtheorem{definition}[theorem]{Definition}
\newtheorem{example}[theorem]{Example}
\newtheorem{remark}[theorem]{Remark}

\newtheorem{assumption}[theorem]{Assumption}

\newcommand{\F}{\mathcal F}
\newcommand{\G}{\mathcal G}
\newcommand{\AW}{\mathcal A\mathcal W}
\newcommand{\Law}{\mathrm{Law}}

\newcommand{\sign}{\mathrm{sign}}

\newcommand{\W}{\mathcal W}
\newcommand{\N}{\mathbb N}
\newcommand{\R}{\mathbb R}

\newcommand{\A}{\mathcal A}
\newcommand{\B}{\mathcal B}
\newcommand{\Pc}{\mathcal P}

\newcommand{\Fc}{\mathcal{F}}

\newcommand{\cpl}{\mathrm{Cpl}}

\newcommand{\cplc}{\mathrm{Cpl}_{\mathrm{c}}}
\newcommand{\cplba}{\cpl_{\mathrm{bc}}}
\newcommand{\cplbc}{\mathrm{Cpl}_{\mathrm{bc}}}

\newcommand{\kr}{\mathrm{KR}}
\newcommand{\antitone}{\mathrm{AT}}
\newcommand{\sync}{\mathrm{sync}}
\newcommand{\async}{\mathrm{async}}

\newcommand{\CW}{\mathcal C\mathcal W}
\newcommand{\SCW}{\mathcal S\mathcal C\mathcal W}

\newcommand{\SW}{\mathcal{SW}}

\newcommand{\D}{\mathop{}\!\mathrm{d}}
\newcommand{\dx}{\mathop{}\!\mathrm{d}x}
\newcommand{\dy}{\mathop{}\!\mathrm{d}y}
\newcommand{\dt}{\mathop{}\!\mathrm{d}t}
\newcommand{\di}{\mathop{}\!\mathrm{d}}

\newcommand{\ind}[1]{\mathds{1}_{\!\left\{{#1}\right\}}}

\def\P{{\mathbb P}}
\def\E{{\mathbb E}}

\def\Q{{\mathbb Q}}
\def\R{{\mathbb R}}

\usepackage{relsize}
\def\fcmp{\mathbin{\raise 0.6ex\hbox{\oalign{\hfil$\scriptscriptstyle \mathrm{o}$\hfil\cr\hfil$\scriptscriptstyle\mathrm{9}$\hfil}}}}

\usepackage{pgfplots}
\usetikzlibrary{arrows.meta}

\pgfplotsset{compat=newest,
    width=6cm,
    height=3cm,
    scale only axis=true,
    max space between ticks=25pt,
    try min ticks=5,
    every axis/.style={
        axis y line=left,
        axis x line=bottom,
        axis line style={line width = 0.3pt,-,>=latex, shorten >=-.4cm}
    },
    every axis plot/.append style={thick},
    tick style={black, thin}
}
\tikzset{
    semithick/.style={line width=0.5pt},
}
\usepgfplotslibrary{groupplots}
\usepgfplotslibrary{dateplot}

\newcommand{\incfig}[1]{%
    \def\svgwidth{0.75\textwidth}
\begingroup%
  \makeatletter%
  \providecommand\color[2][]{%
    \errmessage{(Inkscape) Color is used for the text in Inkscape, but the package 'color.sty' is not loaded}%
    \renewcommand\color[2][]{}%
  }%
  \providecommand\transparent[1]{%
    \errmessage{(Inkscape) Transparency is used (non-zero) for the text in Inkscape, but the package 'transparent.sty' is not loaded}%
    \renewcommand\transparent[1]{}%
  }%
  \providecommand\rotatebox[2]{#2}%
  \newcommand*\fsize{\dimexpr\f@size pt\relax}%
  \newcommand*\lineheight[1]{\fontsize{\fsize}{#1\fsize}\selectfont}%
  \ifx\svgwidth\undefined%
    \setlength{\unitlength}{565.85126261bp}%
    \ifx\svgscale\undefined%
      \relax%
    \else%
      \setlength{\unitlength}{\unitlength * \real{\svgscale}}%
    \fi%
  \else%
    \setlength{\unitlength}{\svgwidth}%
  \fi%
  \global\let\svgwidth\undefined%
  \global\let\svgscale\undefined%
  \makeatother%
  \begin{picture}(1,0.51656603)%
    \lineheight{1}%
    \setlength\tabcolsep{0pt}%
    \put(0,0){\includegraphics[width=\unitlength,page=1]{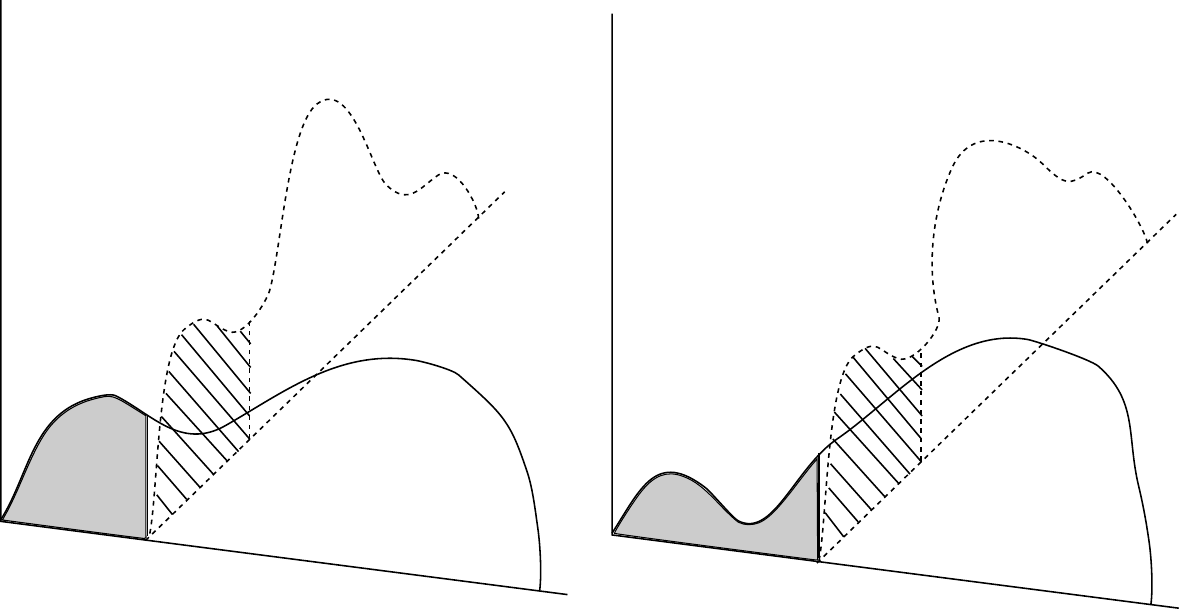}}%
    \put(0.04033325,0.18804548){\color[rgb]{0,0,0}\makebox(0,0)[lt]{\lineheight{1.25}\smash{\begin{tabular}[t]{l}$\mu_1$\end{tabular}}}}%
    \put(0.1162347,0.25335455){\color[rgb]{0,0,0}\makebox(0,0)[lt]{\lineheight{1.25}\smash{\begin{tabular}[t]{l}$\mu_{x_1}$\end{tabular}}}}%
    \put(0.52580089,0.12361979){\color[rgb]{0,0,0}\makebox(0,0)[lt]{\lineheight{1.25}\smash{\begin{tabular}[t]{l}$\nu_1$\end{tabular}}}}%
    \put(0.69082722,0.22830631){\color[rgb]{0,0,0}\makebox(0,0)[lt]{\lineheight{1.25}\smash{\begin{tabular}[t]{l}$\nu_{y_1}$\end{tabular}}}}%
    \put(0.10406878,0.03897666){\color[rgb]{0,0,0}\makebox(0,0)[lt]{\lineheight{1.25}\smash{\begin{tabular}[t]{l}$x_1$\end{tabular}}}}%
    \put(0.21422044,0.13089705){\color[rgb]{0,0,0}\makebox(0,0)[lt]{\lineheight{1.25}\smash{\begin{tabular}[t]{l}$x_2$\end{tabular}}}}%
    \put(0.67298128,0.02090569){\color[rgb]{0,0,0}\makebox(0,0)[lt]{\lineheight{1.25}\smash{\begin{tabular}[t]{l}$y_1$\end{tabular}}}}%
    \put(0.78020259,0.10744647){\color[rgb]{0,0,0}\makebox(0,0)[lt]{\lineheight{1.25}\smash{\begin{tabular}[t]{l}$y_2$\end{tabular}}}}%
  \end{picture}%
\endgroup%

}

\numberwithin{equation}{section}

\author{Julio Backhoff-Veraguas$^1$}
\address{$^1$Faculty of Mathematics, University of Vienna, Austria}
\email{julio.backhoff@univie.ac.at}
\author{Sigrid K\"allblad$^2$}
\address{$^2$Department of Mathematics, KTH Royal Institute of Technology, Sweden}
\email{sigrid.kallblad@math.kth.se}
\author{Benjamin A.\ Robinson$^3$}
\address{$^3$Department of Statistics, University of Klagenfurt, Austria}
\email{benjamin.robinson@aau.at}
\thanks{$^2$Financial support from the Swedish Research Council (VR) under grant 2020-03449 is gratefully acknowledged.\\$^3$This research was funded in part by the Austrian Science Fund (FWF) [10.55776/Y782], [10.55776/P35519], [10.55776/P34743]. For open access purposes, the author has applied a CC BY public copyright license to any author accepted manuscript version arising from this submission.\\
The authors thank the anonymous referees for their comments, which helped to improve this manuscript.}

\title[Adapted Wasserstein distance for SDEs]{Adapted Wasserstein distance between the laws of SDEs}

\date{\today}

\begin{document}

\begin{abstract}
	We consider the bicausal optimal transport problem between the laws of scalar time-homogeneous stochastic differential equations, and we establish the optimality of the synchronous coupling between these laws. The proof of this result is based on time-discretisation and reveals a novel connection between the synchronous coupling and the celebrated discrete-time Knothe--Rosenblatt rearrangement. We also prove a result on equality of topologies restricted to a certain subset of laws of continuous-time processes. We complement our main results with examples showing how the optimal coupling may change in path-dependent and multidimensional settings.
\end{abstract}

\maketitle

\section{Introduction}
For all their merits, the concepts of weak convergence and Wasserstein distances have proven to be insufficient for applications involving stochastic processes where filtrations and the flow of information play a pivotal role. For instance, neither usual stochastic optimisation problems (such as optimal stopping or utility maximisation) nor Doob--Meyer decompositions behave continuously with respect to these topologies. 
Over the last decades, several approaches have been proposed to overcome these shortcomings; in this paper we focus on one such notion, namely the so-called adapted Wasserstein distance.

More precisely, we study the adapted Wasserstein distance between the 
laws of solutions of one-dimensional stochastic differential equations (SDEs) when the space of continuous functions is equipped with the $L^p$-metric. 
We address this problem by embedding it into a class of bicausal optimal transport problems. 
Imposing fairly general conditions on 
the coefficients of the SDEs, typically amounting to time-homogeneity and mild regularity assumptions, our contribution can be summarised as follows:
\begin{itemize}
\item [(i)] characterisation of the optimal coupling, the so-called synchronous coupling, for a class of bicausal optimal transport problems, which notably includes the adapted Wasserstein distance;
\item[(ii)]  a time-discretisation method allowing us to derive most continuous-time statements from their more elementary discrete-time counterparts; 
\item[(iii)] a stability result for optimisers of bicausal optimal transport problems;
\item[(iv)] a result stating that the topology induced by the adapted Wasserstein distance coincides with several topologies (including the weak topology) when restricting to SDEs whose coefficients belong to equicontinuous families;
\item[(v)] examples illustrating the sub-optimality of the synchronous coupling for path-dependent SDEs and in higher dimensions. 
\end{itemize}

A further significant contribution is the connection of two hitherto unrelated objects: the \emph{synchronous coupling} of SDEs, which is the coupling that arises when a single Wiener process drives two SDEs; and the \emph{Knothe--Rosenblatt} rearrangement, which is a celebrated discrete-time adapted coupling that preserves the lexicographical order. One key result that we establish is an optimality property of the Knothe--Rosenblatt rearrangement. 
We then argue that in a certain sense, the synchronous coupling is the continuous-time counterpart of the Knothe--Rosenblatt rearrangement.

The pioneering work of \citet{BiTa19} contains a similar statement to our contribution (i) above; they use PDE techniques to prove optimality of the synchronous coupling for the problem of optimally controlling the correlation between one-dimensional SDEs with smooth coefficients. 
Our first main result establishes that, for general cost functions and path-dependent SDEs, bicausal optimal transport problems admit such a control reformulation.\footnote{Note added in revision: See also~\citet{CoLi24} for an extension of this result to higher dimensions.}  
A posteriori, it is thus clear that \cite{BiTa19} establishes (i) for the adapted $2$-Wasserstein distance and smooth coefficients. 
Using probabilistic methods, we generalise the optimality result of \cite{BiTa19} to more general cost functions and SDEs.

\subsection*{Adapted Wasserstein distance}

		We now define the adapted Wasserstein distance and give some motivation for its introduction. Let $\Omega=C([0,1],\R)$ be the space of continuous paths from $[0, 1]$ into $\R$ endowed with the uniform topology and corresponding Borel $\sigma$-field $\mathcal B(\Omega)$, and endow the product space $\Omega\times\Omega$ with the corresponding product $\sigma$-field. Write $\omega$ and $\bar{\omega}$ for the first and second components of the canonical process on $\Omega\times\Omega$.
		For any two probability measures $\mu,\nu$ on $\Omega$, the set of \emph{couplings} between $\mu$ and $\nu$, $\cpl(\mu,\nu)$, consists of all probability measures $\pi$ on $\Omega\times\Omega$ with \emph{marginals} $\mu$, $\nu$; that is, $\pi(B,\Omega)=\mu(B)$ and $\pi(\Omega,B)=\nu(B)$, for all $B\in\mathcal{B}(\Omega)$.
		For $p\ge 1$, we write $\Pc_p$ for the set of probability measures $\mu$ on $\Omega$ with finite $p$\textsuperscript{th} moments with respect to the $L^p$-norm on $\Omega$; i.e.~$\int_{\Omega}\int_0^1|\omega_t|^p \D t \D\mu < \infty$. 
		
		The classical $p$-Wasserstein distance $\W_p$ on $\Pc_p$ with respect to the $L^p$-norm on the underlying space $\Omega$ then takes the following form (see, for example, Villani \cite[Definition 6.1]{Vi09}): 
\begin{equation}\label{eq:wasserstein_Lp}
	\W_p^p(\mu,\nu) \coloneqq \inf_{\pi\in\cpl(\mu,\nu)}\E^\pi\!\left[\int_0^1\!\left|\omega_t-\bar{\omega}_t\right|\!^p\dt\right],
	\qquad\mu,\nu\in\Pc_p.
\end{equation}
	This distance notably fails to take the flow of information into account. For example, the values of optimisation problems for continuous-time stochastic processes may not be continuous in Wasserstein distance with respect to the reference measure; see \Cref{ex:motivation}.	
	
	As a remedy, the adapted Wasserstein distance is defined by restricting to couplings that respect the asymmetric flow of information originating from the processes. 
	Following \cite{BaBaBeEd19a}, we define bicausal couplings as follows: Let $\Fc$ and $\bar \Fc$ be the natural filtrations of $\omega$ and $\bar \omega$, respectively. For any $t \in [0, 1]$ and any probability measure $\mu$ on $\Omega$, write $\Fc_t^\mu$ for the completion\footnote{I.e.\ $\F^\mu_t$ is the sigma-algebra generated by $\F_t$ and the null sets for $\mu$.} of $\Fc_t$ under $\mu$. For any $\pi\in\cpl(\mu,\nu)$, let $\pi_\omega(\di\bar\omega)$ be the regular disintegration kernel for which $\pi(\di\omega,\di\bar\omega)=\mu(\di\omega)\pi_\omega(\di\bar\omega)$.
	
\begin{definition}[bicausal couplings]\label{def:bicausal_couplings}
The set of \emph{causal couplings} $\cplc(\mu,\nu)$ consists of all $\pi\in\cpl(\mu,\nu)$ such that, for each $t\in [0, 1]$ and $A\in\bar\Fc_t$, 
\begin{equation}
	\textrm{$\omega\mapsto \pi_\omega(A)$
	 is $\Fc_t^\mu$-measurable}.
\end{equation}
The set of \emph{bicausal couplings} $\cplbc(\mu,\nu)$ consists of all $\pi\in\cplc(\mu,\nu)$ with $S_\#\pi\in\cplc(\nu,\mu)$, where $S(\omega,\bar\omega)=(\bar\omega,\omega)$, for all $(\omega,\bar\omega)\in\Omega\times\Omega$.
\end{definition}

Put into words: `one cannot look into the future when deciding where to allocate mass at a given time'. This emphasises the role played by the flow of information; i.e.\ filtrations.\footnote{The intuition behind the concept of causality is perhaps most easily grasped in a discrete-time setup, i.e.\ when considering a finite set of time points, say $\{1,2,\dots,N\}$. The defining property of causality can then be phrased as requiring, with obvious adaptation of notation, that $\pi((\bar\omega_1,...,\bar\omega_n)\in A|\omega_1,...,\omega_N)=\pi((\bar\omega_1,...,\bar\omega_n)\in A|\omega_1,...,\omega_n)$, for all $A\in\mathcal{B}(\R^n)$, $n=1,...,N$. In such a discrete-time setting, if the coupling is further supported on the graph of a function, say $\varphi \colon \R^N\to\R^N$ (i.e.\ a Monge map), then causality amounts to $\varphi(x_1,...,x_N)=(\varphi_1(x_1),\varphi_2(x_1,x_2),\dots,\varphi_N(x_1,\dots,x_N))$, for some functions $\varphi_n \colon \R^n\to\R$, $n=1,...,N$.}
	With the above notation at hand, we are now ready to define the \emph{adapted Wasserstein distance} $\AW_p$, $p\ge 1$:
\begin{equation}\label{eq:adapted_wasserstein}
	\AW_p^p(\mu,\nu) \coloneqq \inf_{\pi\in\cplbc(\mu,\nu)}\E^\pi\!\left[\int_0^1\!\left|\omega_t-\bar{\omega}_t\right|\!^p\dt\right],
	\qquad\mu,\nu\in\Pc_p.
\end{equation}

	To give an historical account, the condition of causality can be traced back, at least, to the work on existence of solutions of SDEs by \citet{YW}; it has also appeared under the name of \emph{compatibility} in \citet{Ku07}. The concept was recently popularised and studied in a continuous-time framework by \citet{La18}, and systematically investigated for discrete-time processes using dynamic programming arguments in \cite{BaBeLiZa17} (see also \cite{Ka17} for a recursive approach to a closely related optimal stopping problem).
	We refer to \citet{BeLa20} for further historical remarks and for an account of the connections to the filtering literature.
	We also refer to \citet{BeLa20,BePaSc21c} for a detailed exposition of how Monge maps relate to general transport plans in the presence of causality constraints. 

	To the best of our knowledge, the symmetric condition of bicausality first appeared in the work of \citet{Ru85}; for a more recent account we refer again to \cite{BaBeLiZa17}. A distance based on the bicausality condition was independently introduced and studied, under the name of \emph{nested distance}, in a series of papers by Pflug and Pichler; see, for example, \cite{PfPi12} and \cite{PfPi14} and the references therein. 
	The concept of causality aside, numerous alternative approaches to incorporating the flow of information into process distances can be found in the literature. Most notably, albeit in different ways, the seminal works of \citet{Al81} and \citet{He96} both rely on incorporating the distance between certain conditional disintegration kernels of the processes. See also \cite{BaBaBeEd19b}, where it was shown that these different distances all generate the same topology. 
	
	In continuous time, for diffusion processes, the modified Wasserstein distance was introduced in \cite{BiTa19}; we show herein that this distance coincides with the adapted Wasserstein distance.
	For general continuous semi-martingales, motivated by financial applications, an adapted Wasserstein distance was defined in \cite{BaBaBeEd19a} with respect to a cost function that compares the drift and martingale parts of the Doob--Meyer decomposition separately.
	We refer to \cite{AcBaCa18,AcBaZa20} for further studies of adapted distances in continuous time; see also \cite[Section~2]{BaBeLiZa17} for an exposition of the related literature within mathematical finance.

\subsection*{Optimality of the synchronous coupling}

	Throughout this article, we consider the laws of solutions of SDEs of the following type: 
\begin{equation}\label{eq:sde}
	\di X^{b,\sigma}_t=b(X^{b,\sigma}_t)\dt+\sigma(X^{b,\sigma}_t)\di W_t,
	\quad X^{b,\sigma}_0=x_0,
	\quad t \in [0, 1],
\end{equation}
where $b \colon \R\to\R$ and $\sigma \colon \R\to\R_+ = [0, \infty)$ are some measurable functions such that a unique strong solution $X^{b, \sigma}$ exists; we write $\mu^{b,\sigma} \coloneqq  \Law(X^{b, \sigma})$ for the induced probability measure on $\Omega$. We suppose that all SDEs are equipped with the same initial condition, $x_0\in\R$, and omit it from the notation.
	
	Given two such measures, $\mu^{b,\sigma},\mu^{\bar b,\bar\sigma}$, we can couple them as follows: Let $(\Omega,\Fc,\P)$ be a probability space supporting a Wiener process $W$, and let $X^{b,\sigma},X^{\bar b,\bar\sigma}$ be the solutions of \eqref{eq:sde} with coefficients $(b, \sigma)$ and $(\bar b, \bar \sigma)$, respectively, when both SDEs are driven by $W$. In this way, we define a bicausal coupling $\P\circ (X^{b,\sigma},X^{\bar b,\bar\sigma})^{-1} \in \cplba(\mu^{b, \sigma}, \mu^{\bar b, \bar \sigma})$. This coupling plays a pivotal role throughout the article and we name it the \emph{synchronous coupling}. 
	
	Our first main result establishes general conditions under which this coupling is optimal. 
	
	\begin{assumption}\label{ass:main}
		The coefficients $b\colon \R \to \R$ and $\sigma\colon \R \to \R_+$ in \eqref{eq:sde} are continuous, have linear growth, and are such that pathwise uniqueness holds for \eqref{eq:sde}.
	\end{assumption}

	\begin{theorem}\label{thm:synchronous_optimal}
		Suppose that $(b, \sigma)$ and $(\bar b, \bar \sigma)$ satisfy \Cref{ass:main}. Then, for any $p\ge 1$, the synchronous coupling attains the infimum in \eqref{eq:adapted_wasserstein} defining $\AW_p(\mu^{b,\sigma},\mu^{\bar b,\bar\sigma})$. 
	\end{theorem}
	
	Our main result notably establishes the optimality of the synchronous coupling not only for the adapted Wasserstein distance, but also for the \emph{bicausal optimal transport} problem with respect to a more general class of cost functions (see \Cref{thm:optimality_general_c}); we also provide this conclusion under a different set of assumptions allowing for the drift coefficients to be discontinuous (see \Cref{prop:zvonkin}). The follow-up work \cite{RoSz24} further extends the main result in the direction of SDEs with irregular coefficients.
 
 	For one-dimensional SDEs with sufficiently regular coefficients, the optimality of the synchronous coupling was first established in \cite{BiTa19} for the so-called modified Wasserstein distance, which is the distance obtained when optimising the cost in \eqref{eq:adapted_wasserstein} over couplings induced by solutions of a pair of SDEs \eqref{eq:sde} driven by \emph{correlated} Wiener processes. 
 	It then follows from the above result combined with \cite[Section 2.1]{BiTa19} that the adapted and modified Wasserstein distances coincide.
 	A crucial part of our analysis is the a priori reformulation of the adapted Wasserstein distance in terms of an associated control problem (where one controls the degree of correlation); a similar result holds for a bicausal optimal transport problem with a general cost function when the marginals are given by possibly path-dependent SDEs (see \Cref{prop:bion-nadal-talay-bc}). 
	 	
 	In \cite{BiTa19}, the authors take a stochastic control approach and their proofs rely on a verification argument for the associated Hamilton--Jacobi--Bellman (HJB) equation. 
 	While such stochastic control arguments provide the natural continuous-time analogue of the recursive arguments used to prove optimality of the Knothe--Rosenblatt rearrangement in discrete time (see \Cref{rem:kr-dpp-proof}), the use of classical solutions of the HJB equation, as employed in \cite{BiTa19}, inevitably requires the cost function as well as the coefficients of the SDEs to be smooth enough for the associated stochastic flows to be differentiable. 
	Here, we rather take a probabilistic approach to prove the optimality of the synchronous coupling. This approach enables us to relax the assumptions on the cost function and the coefficients and establish this optimality property in its natural generality.
	
	We clarify that, while the analysis in \cite{BiTa19} also pertains to multidimensional diffusions, the authors only identify an explicit optimiser in dimension one. In fact, the optimality of the synchronous coupling does not generally extend to higher dimensions. Our \Cref{ex:2d-1,ex:2d-2} demonstrate the sub-optimality of the synchronous coupling for certain multidimensional SDEs. We also note that the discrete approximation methods described below rely on one-dimensional results from optimal transport; see \Cref{rem:monotone-optimality}.

\subsection*{Discrete approximation methods and stability}
	
	Consider now the problem of optimally coupling the laws of one-dimensional discrete-time stochastic processes with $n \in \N$ time steps, i.e.~coupling probability measures on $\R^n$.
	A key object of study in this paper is the \emph{Knothe--Rosenblatt rearrangement}, introduced in \cite{Ro52} and \cite{Kn57}, which generalises the classical monotone rearrangement to the laws of such discrete-time processes; see \Cref{fig:kr} for an illustration.
	When restricting to bicausal couplings and imposing certain monotonicity properties on the marginal laws, it turns out that the Knothe--Rosenblatt rearrangement is optimal for an $L^p$-cost.
	For a two-step discrete bicausal optimal transport problem, this was first shown by R{\"u}schendorf \cite{Ru85}. In \cite{BaBeLiZa17}, based on a recursive argument, the result was then generalised to a multi-stage discrete problem with Markov marginal laws. We generalise this result to an even broader class of cost functions and marginal laws, and we link the assumptions on the marginals to the notion of stochastic monotonicity; see \Cref{prop:kr-discrete-optimality}. We emphasise that the dimension $n \in \N$ here represents the number of time steps of the discrete-time process and that the state space is necessarily one-dimensional; see \Cref{rem:monotone-optimality}.
	
	This discrete-time optimality result underpins our analysis in continuous time; applying this result to a carefully chosen discretisation of the SDEs, we deduce the optimality of the synchronous coupling. 
	Indeed, our proof relies on an approximation procedure where we first solve the associated discrete-time problem and then pass to the limit. 
	Our method of proof thus unveils the informational and structural similarities between the Knothe--Rosenblatt and synchronous couplings. For this reason, we advocate the interpretation of the synchronous coupling as the continuous analogue of the Knothe--Rosenblatt rearrangement.
	In order to carry out the above procedure, for the class of bicausal transport problems that we study, we establish a stability result for optimal couplings (see \Cref{prop:stability_optimiser}).

	\begin{figure}[h]
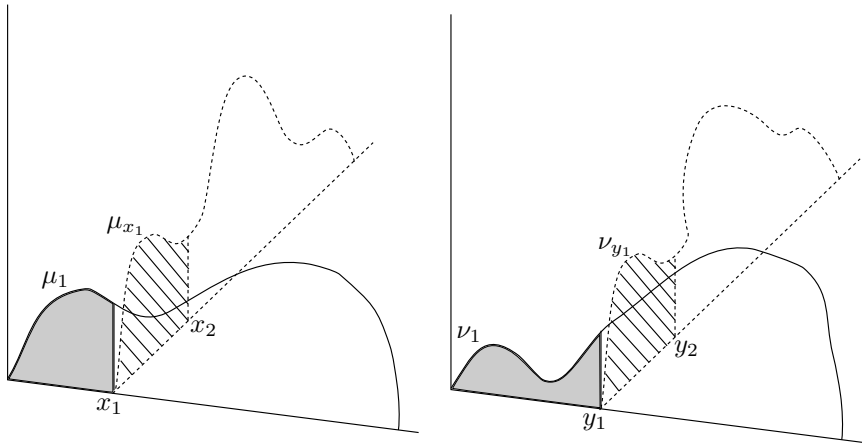

	    \centering
    	\incfig{fig-kr}
	    \caption{Illustration of the Knothe--Rosenblatt rearrangement in two dimensions. The first marginals of $\mu$, $\nu$ are denoted $\mu_1$, $\nu_1$, and the conditional distributions by $\mu_{x_1}$, $\nu_{y_1}$. Similarly shaded regions have the same area.}
    	\label{fig:kr}
	\end{figure}
	
	We further obtain approximation results in the adapted Wasserstein distance for bicausal couplings between the laws of SDEs. Given coefficients $(b,\sigma)$, $(\bar b,\bar\sigma)$ and $n \in \N$, consider the \emph{Euler--Maruyama scheme} $(X^n,\bar X^n)$, which is given by $(X^n_0, \bar X^n_0) = (x_0, x_0)$ and, for $h=1/n$, $k=0,...,n-1$,
\begin{equation}\label{eq:approx}
\begin{cases}
    X^n_t = X^n_{kh}+b(X^n_{kh})\!\left(t-kh\right)+\sigma(X^n_{kh})\!\left(W_t-W_{kh}\right)\\
	\bar X^n_t = \bar X^n_{kh}+\bar b(\bar X^n_{kh})\!\left(t-kh\right)+\bar\sigma(\bar X^n_{kh})\!\left(\bar W_t-\bar W_{kh}\right)
\end{cases}
t\in(kh,(k+1)h].
\end{equation}
With the adapted Wasserstein distance defined analogously to \eqref{eq:adapted_wasserstein} for marginal distributions on $\Omega\times\Omega$, we have the following result: 
\begin{theorem}\label{thm:convergence}
	Let $b,\bar b\colon \R \to \R$, $\sigma,\bar\sigma \colon \R\to\R_+$ be Lipschitz, and let $\pi\in\cplbc(\mu^{b,\sigma},\mu^{\bar b,\bar\sigma})$. Then there exists a probability space supporting two correlated Wiener processes $W$ and $\bar W$ such that the joint law of the processes $(X^n,\bar X^n)$ given by \eqref{eq:approx} satisfies $\AW_p(\Law(X^n,\bar X^n),\pi)\xrightarrow{n \to \infty} 0$, $p\ge 1$. 
\end{theorem}

The above result holds also for SDEs with path-dependent coefficients (see \Cref{thm:convergence-general}). To obtain our main results, we will work with a modification of the Euler--Maruyama scheme in which the increments of the Wiener process are truncated; cf.~\citet{MiReTr02,LiPa22}. Adapting \Cref{thm:convergence} for this modified scheme, we will prove convergence of the Knothe--Rosenblatt rearrangement to the synchronous coupling in a sense that is made precise in \Cref{prop:aw-conv-kr-sync}.

While these results on discretisation and stability are crucial for our analysis, they are also of independent interest. For instance, our time-discretisation method suggests a possible approach to numerical approximation of adapted Wasserstein distances between the laws of SDEs; however, we do not explore this direction further. 
For existing numerical methods for computing adapted Wasserstein distances, we refer to \citet{EcPa22,PiWe21}, and the references therein.

\subsection*{The synchronous distance and the associated topology}

	We also study the topology induced by the adapted Wasserstein distance $\AW_p$.
	In particular, we investigate the relationship between the topologies induced by different distances on the spaces $\Pc_p$ of laws of continuous-time stochastic processes.
	
	The classical Wasserstein distance  \eqref{eq:wasserstein_Lp} metrises the usual weak topology on the space $\Pc_p$ (see, e.g.~\citep[Theorem 7.12]{Vi03}).
	Moreover, for $p \geq 1$, one may consider the (asymmetric) \emph{causal} Wasserstein distance $\CW_p(\mu,\nu)$, $\mu, \nu \in \Pc_p$, defined analogously to the adapted (bicausal) Wasserstein distance, by replacing $\cplbc(\mu,\nu)$ with $\cplc(\mu,\nu)$ in \eqref{eq:adapted_wasserstein}. In this asymmetric setting, we say that $\mu_n$ converges to $\mu$ in $\CW_p$, if $\CW_p(\mu,\mu_n)\to 0$. We also consider its symmetrised version $\SCW_p(\mu,\nu)=\max(\CW_p(\mu,\nu),\CW_p(\nu,\mu))$, $\mu,\nu\in\Pc_p$.
	Finally, inspired by the pivotal role played by the synchronous coupling, we introduce the \emph{synchronous distance} $\SW_p$, defined by
	\begin{equation}
		\SW_p^p\big(\mu^{b,\sigma},\mu^{\bar b,\bar\sigma}\big)
		 \coloneqq \E\!\left[\int_0^1\big|X^{b,\sigma}_t-X^{\bar b,\bar\sigma}_t\big|^p \D t \right], \quad p \geq 1,
	\end{equation}
	where $X^{b,\sigma},X^{\bar b,\bar\sigma}$ are the $p$-integrable solutions of the SDE \eqref{eq:sde}, with coefficients $(b, \sigma),(\bar b, \bar \sigma)$, evaluated on some probability space with respect to the \emph{same} Wiener process $W$ (c.f.~the definition of the synchronous coupling). 
	This distance is notably stronger than all of the above-mentioned distances.
		
	We show that all the distances discussed above induce the same topology when restricted to solutions of the SDE \eqref{eq:sde} for which the coefficients belong to the following set:
\begin{equation}
	\A^\Lambda
	=
	\!\left\{\varphi\in C(\R,\R)\colon
	 |\varphi(x)-\varphi(y)|\le \Lambda |x-y|
	 \textrm{ and }
	 |\varphi(0)|\le \Lambda,
	 \;x,y\in\R
	 \right\},
	 \;\; \Lambda>0.
\end{equation}

\begin{theorem}\label{thm:topologies}
	Restricted to the set 
	$\Pc^\Lambda=\{
 	\mu^{b,\sigma}:b,\sigma\in\A^\Lambda\}$, $\Lambda>0$,
	the topologies induced by the following metrics all coincide and are independent of $p\in[1,\infty)$: 
	\begin{itemize}[label = ---]
		\item $\SW_p$, the synchronous distance;
		\item $\AW_p$, the adapted Wasserstein distance;
		\item $\SCW_p$, the symmetrised causal Wasserstein distance;
		\item $\W_p$, the Wasserstein distance.
	\end{itemize}
	The above topologies also remain equal when, in the definition of any of the metrics, we replace the cost $\int_0^1 |\omega_t - \bar \omega_t|^p \D t$ by $\sup_{t \in [0, 1]}|\omega_t - \bar \omega_t|^p$.
	
	This common topology is further equal to the topology of $\CW_p$ convergence, the topology of weak convergence when we equip $\Omega$ with the $L^p(\D t)$ norm, for arbitrary $p \in [1, \infty]$, and also to the topology of convergence in finite-dimensional distributions. 
	
	Moreover, $\Pc^\Lambda$ is compact in this common topology.  
\end{theorem}

	Note that, in contrast to the optimality results presented thus far, the topological equivalence of \Cref{thm:topologies} carries over to higher dimensions; see \Cref{rem:multidim-topo}.
	
	In discrete time, \citet{BePaPo21} studied a distance induced by the Knothe--Rosenblatt rearrangement and showed that it is topologically equivalent to the adapted Wasserstein distance. Our \Cref{thm:topologies} can thus be seen as a continuous-time analogue of \cite[Theorem 1.4]{BePaPo21}, with the synchronous coupling in place of the Knothe--Rosenblatt rearrangement.
	For discrete-time processes, \cite{BaBaBeEd19b} also established that the same topology is generated by the adapted Wasserstein distance, the nested distance, and the distances introduced by \citet{Al81} and \citet{He96}; we refer to \cite{BaBeEdPi17,Ed19} for further properties of this common topology.
	In fact, an even stronger result is true: \citet{BaBePa21} and \citet{Pa22} go beyond the convention of identifying a process with its law and instead consider processes equipped with a filtration. It is shown that all topologies that are strong enough to encode the information of the filtration still coincide, even in this generalised discrete-time setting.
	In a continuous-time setup, \citet[Propositions 1.8, 1.9]{BiTa19} give a result in a similar spirit to \Cref{thm:topologies} for their modified Wasserstein distance. We note that our proof is remarkably simple as it is a straightforward application of stability results for SDEs.

\subsection*{Structure of the article}

In \Cref{sec:formulation}, we establish the equality of the modified and adapted Wasserstein distances; we also provide a stability result and prove \Cref{thm:convergence}.
We give our optimality result for the Knothe--Rosenblatt rearrangement in \Cref{sec:kr}. In \Cref{sec:numerical-scheme}, we introduce a variation of the classical Euler--Maruyama scheme and use it to prove \Cref{thm:synchronous_optimal} under an additional Lipschitz assumption. We complete the proof of \Cref{thm:synchronous_optimal} in \Cref{sec:continuous-general}. In \Cref{sec:discontinuous}, we establish a variation of \Cref{thm:synchronous_optimal} which allows for discontinuities in the drift coefficients.
In \Cref{sec:topology}, we prove \Cref{thm:topologies} on the equality of topologies.
In \Cref{sec:discussion}, we collect some examples, the first of which motivates the introduction of the adapted Wasserstein distance. We then present counterexamples to the optimality of the synchronous coupling in non-Markovian and higher dimensional settings, as well as a counterexample to the optimality of the Knothe--Rosenblatt rearrangement for a particular choice of cost function. Various approximation and stability results are deferred to the appendix.

\section{Preliminary results on bicausal couplings and approximation in $\AW_p$}\label{sec:formulation}

Throughout this paper we work in dimension one. 
Define the space of continuous paths $\Omega\coloneqq  C([0, 1], \R)$ equipped with the uniform topology and corresponding Borel sigma-field $\F \coloneqq \B(\Omega)$.
We also equip $\Omega$ with the canonical filtration $(\F_t)_{t \in [0, 1]}$ and note that $\F = \F_1$.
 
We will repeatedly make use of the following definition of correlated Wiener processes.

\begin{definition}[correlated Wiener process]\label{def:correlated}
	Let $(\Omega, \F, (\G_t)_{t \in [0, 1]}, \P)$ be a complete filtered probability space on which two standard real-valued $\G$-Wiener processes $W$ and $\bar W$ are defined. We say that the two-dimensional process $(W, \bar W)$ is a $\rho$\emph{-correlated Wiener process} if the cross-variation satisfies
	\begin{equation}
		\D \langle W, \bar W \rangle_t = \rho(t, W, \bar W) \D t,
	\end{equation}
    for some $\B([0,1])\otimes \G$-progressively measurable function $\rho\colon [0, 1] \times \Omega \times \Omega \to [-1, 1]$, which is then unique in the $(d\mathbb P\times dt)$-sense. We say that  $(W, \bar W)$ is a correlated Wiener process if it is a $\rho$-correlated Wiener process for some $\rho$. 
\end{definition}

Given $x_0,\bar x_0 \in \R$ and path-dependent coefficients $b, \bar b\colon [0, 1] \times \Omega \to \R$, $\sigma, \bar \sigma\colon [0, 1] \times \Omega \to \R_+$, which are progressively measurable, consider the system of SDEs
\begin{equation}\label{eq:sde-path-dep}
	\begin{cases}
  X_t = x_0 + \int_0^t b(s,X)\D s + \int_0^t \sigma(s,X) \D W_s\\
  \bar X_t = \bar x_0 + \int_0^t\bar b(s,\bar X)\D s + \int_0^t \bar \sigma(s,\bar X)\D \bar W_s
\end{cases}
t \in [0, 1].
\end{equation}
Suppose that there exists some correlated Wiener process $(W, \bar W)$ and a $\G$-adapted process $(X, \bar X)$ (cf.\ \Cref{def:correlated}) that satisfies \eqref{eq:sde-path-dep}. In this case, we say that $(X, \bar X)$ is a \emph{strong solution of the system} \eqref{eq:sde-path-dep} driven by the correlated Wiener process $(W, \bar W)$. Writing $\mu, \nu$ for the marginal laws of $X, \bar X$, we denote the set of couplings of the form $\Law(X, \bar X)$ by $\widetilde\cpl (\mu, \nu)$.

\subsection{Characterisation of bicausal couplings between SDEs}

\begin{proposition}\label{prop:bion-nadal-talay-bc}
	Suppose that there exist unique strong solutions to the SDEs \eqref{eq:sde-path-dep} and write $\mu,\nu$ for their respective laws. Then the set of bicausal couplings $\cplbc(\mu, \nu)$ is equal to the set of couplings $\widetilde\cpl(\mu, \nu)$.
\end{proposition}

\begin{proof}
	Suppose that $\tilde \pi \in \widetilde\cpl(\mu, \nu)$. Then $\tilde \pi \in \cpl(\mu, \nu)$. Moreover, there is some correlated Wiener process $(W, \bar W)$ such that $\tilde \pi = \Law(X, \bar X)$, where $(X, \bar X)$ is a strong solution of the system \eqref{eq:sde-path-dep} driven by $(W, \bar W)$. We see that $\Law(W, \bar W)$ is bicausal, since $W$, $\bar W$ are Wiener processes with respect to the same filtration $\G$. Since $(X, \bar X)$ is adapted to the completed filtration of $(W, \bar W)$, we conclude that $\tilde \pi$ is bicausal.
		
	To show the converse, suppose now that $\pi \in \cplba(\mu,\nu)$. {Let $(\omega, \bar \omega)$ be the canonical process on the product space, and let $\F, \bar \F$ denote the canonical filtrations corresponding to $\omega, \bar \omega$, respectively.
	Note that, by definition of the measure $\mu$, we may define a continuous process $M$ such that, under $\pi$, $M$ is an $\F^\mu$-martingale with quadratic variation $\langle M \rangle$ given by $\langle M\rangle_t=\int_0^t \sigma(s,\omega)^2\D s$, for all $t \in [0, 1]$, and
	\begin{equation}\label{eq:causal-martingale-rep}
		\omega_t = x_0 + \int_0^t b(s,\omega)\D s + M_t, \quad \text{for all } t \in [0, 1].
	\end{equation}
	Since the coupling $\pi$ is bicausal, we have the following independence property: for any $s \in [0, 1]$, conditional on $\F^\mu_s$, $(\omega_t)_{t\in [0, 1]}$ is independent of $\bar\F_s$ under $\pi$. Therefore, for any $s, t \in [0, 1]$ with $s < t$, we have
	\begin{equation}\label{eq:filtration-decomp-causality}
		\begin{split}
			\E^\pi\!\left[M_t - M_s \mid \F^\mu_s \otimes \bar\F_s\right] & = \E^\pi\!\left[M_t - M_s \mid \F^\mu_s \right] = 0,
		 \end{split}
	\end{equation}
	and by the tower property of conditional expectation, $\E^\pi\!\left[M_t - M_s \mid \F_s \otimes \bar\F_s\right] = 0$. Thus $M$ is also an $\F \otimes \bar \F$-martingale under $\pi$.
	
	Using the symmetry of the definition of bicausality, we find that $\bar \omega$ admits an analogous representation to \eqref{eq:causal-martingale-rep} under $\pi$ for some continuous $\F \otimes \bar \F$-martingale $\bar M$.

	Enlarging the probability space as necessary, introduce a Wiener process $\hat W$ that is independent of $M$ and $\bar M$ under $\pi$. Define processes $W, \bar{W}$ via
		\begin{equation}
			\D W_t = \mathds1_{\{\sigma(t,\omega)=0\}}\D\hat W_t+\mathds1_{\{\sigma(t,\omega)\neq 0\}} \frac{\D M_t}{\sigma(t,\omega)},
			\quad
			\D \bar{W}_t = \mathds1_{\{\bar\sigma(t,\bar \omega)=0\}}\D\hat W_t+\mathds1_{\{\bar \sigma(t,\bar \omega)\neq 0\}}\frac{\D \bar M_t}{\bar{\sigma}(t,\bar \omega)},
		\end{equation}
	and let $\G$ denote the completion of the natural filtration of $(W, \bar W)$ under $\pi$.
	Then $(\omega, \bar \omega)$ satisfies \eqref{eq:sde-path-dep} driven by $(W, \bar{W})$. Further, by the bicausality of $\pi$ and the independence of $\hat W$, both $W$ and $\bar W$ are $\G$-martingales under $\pi$. It then follows from  L\'evy's characterisation that both $W$ and $\bar W$ are $\G$-Wiener processes. By the Kunita--Watanabe inequality (see, e.g.~\cite[Proposition 4.18]{LG16}), we have that $\D \langle W, \bar W \rangle$ is almost surely absolutely continuous with respect to Lebesgue measure, and so there exists a $\B([0,1]) \otimes \G$-progressively measurable function $\rho\colon [0, 1] \times \Omega \times \Omega \to [-1, 1]$ such that $\rho_t \D t = \D \langle W, \bar W \rangle_t$, for all $t \in [0, 1]$.
	Therefore $(W, \bar W)$ is a $\rho$-correlated Wiener process, as in \Cref{def:correlated}.
		
	By assumption, the SDE with coefficients $b, \sigma$ driven by the Wiener process $W$ admits a unique strong solution $\tilde X$. Letting $\mathcal H$ be the natural filtration of $\hat W$, we have that both $\tilde X$ and $\omega$ are adapted to $\F \otimes \bar \F \otimes \mathcal H$. By pathwise uniqueness, we deduce that $\tilde X=\omega$ under $\pi$ and, in particular, $\omega$ is adapted to the natural filtration of $W$ and hence to $\G$. The analogous statement applies to $\bar \omega$,  and so we have that the canonical process $(\omega, \bar \omega)$ under the bicausal coupling $\pi$ is a strong solution of the system \eqref{eq:sde-path-dep} with respect to $(W, \bar W)$, as required.}
\end{proof}

\begin{remark}[modified Wasserstein distance]\label{rem:BN-T:distance}
For any $p \ge 1$, consider the subset $\bar \Pc_p \subset \Pc_p$ consisting of laws of SDEs of the form \eqref{eq:sde-path-dep}. On this set, one can define a metric $\bar{\W}_p$ in the same way as the adapted Wasserstein distance $\AW_p$ is defined in \eqref{eq:adapted_wasserstein}, but with the set $\cplbc(\mu, \nu)$ of bicausal couplings between $\mu, \nu \in \bar \Pc_p$ replaced by $\widetilde\cpl(\mu, \nu)$.
According to \Cref{prop:bion-nadal-talay-bc}, for $\mu, \nu \in \bar \Pc_p$, we have $\bar{\W}_p(\mu,\nu) = \AW_p(\mu,\nu)$.

For $p = 2$, restricting to elements of $\bar \Pc_2$ that are laws of time-homogeneous SDEs of the form \eqref{eq:sde}, $\bar{\W}_2$ thus defined is the \emph{modified Wasserstein distance} introduced by Bion-Nadal and Talay in \cite{BiTa19}.
\end{remark}

\begin{remark}
	\label{rem:right.cont.filt}
	In our definition of (bi)causal couplings we make use of the canonical filtration, while in
	 \cite{BaBaBeEd19a} the right-continuous filtration is used.  In general, our definition gives a smaller set of couplings. However, when the marginal processes are strong Markov both sets coincide following an application of \cite[Ch.\ 2, Proposition 7.7]{KaSh91}. A sufficient condition for this to hold is that the coefficients of the SDEs are locally bounded and that the associated martingale problems are well-posed for all initial conditions; see \cite[Ch.\ 5, Theorem 4.20]{KaSh91}. In \Cref{ex:motivation}, we also encounter a situation where the above sets of couplings coincide, although one of the marginal processes is not strong Markov.
\end{remark}

\subsection{Stability of bicausal optimisers}

Our strategy for proving the main results in this paper is to approximate our problem of interest by problems for which the optimiser is known. A key ingredient in this approach is the stability of optimisers. We here establish such a stability result for a more general class of \emph{bicausal optimal transport problems} obtained by replacing the $p$-norm in the definition \eqref{eq:adapted_wasserstein} of $\AW_p$ with a general \emph{cost function} $c\colon \R \times \R \to \R$; we typically assume that there exists some $K>0$ such that
\begin{equation}\label{eq:p_growth_c}
	|c(x,y)|\leq K[1+|x|^p+|y|^p],
	\quad\textrm{for all }x,y\in\R.
\end{equation}

Since it will be useful to allow for approximation by processes that are not necessarily continuous, we take $\hat\Omega=\mathbb{D}([0,1],\R)$ to be the Polish space of c{\`a}dl{\`a}g functions equipped with the Skorokhod topology and corresponding Borel $\sigma$-field $\hat \Fc$. We define product spaces, the canonical process, the set of $p$-integrable probability measures, and the set of couplings in total analogy to the continuous case.
In particular, for $p\ge 1$ and $\pi,\pi'\in\Pc(\hat\Omega\times\hat\Omega)$ with finite $p\textsuperscript{th}$ moment, in analogy to \eqref{eq:wasserstein_Lp}, the $p$-Wasserstein distance is given by
\begin{align}\label{eq:distance_Wp_Lp}
\W_p^p(\pi,\pi')=\inf_{\alpha\in\cpl(\pi,\pi')}\E^\alpha\!\left[\int_0^1|(\omega_t,\bar\omega_t)-(\omega'_t,\bar\omega'_t)|^p\dt\right],
\end{align}
where $\cpl(\pi,\pi')$ denotes the set of probability measures on $(\hat\Omega\times\hat\Omega)\times(\hat\Omega\times\hat\Omega)$ with marginal distribution onto the first (resp. last) two coordinates given by $\pi$ (resp. $\pi'$), and $((\omega,\bar\omega),(\omega',\bar\omega'))$ denotes the corresponding canonical process.

\begin{remark}\label{rem:borel-sets}
	Although we consider Wasserstein distances with respect to an $L^p$-metric, we defined the Borel $\sigma$-fields $\Fc$ and $\hat \Fc$ on $\Omega$ and $\hat \Omega$ with respect to the uniform topology and Skorokhod topology, respectively. We now verify that these Borel $\sigma$-fields are equal to the Borel $\sigma$-fields corresponding to the topology of $L^p$-convergence.  Indeed, both $\Fc$ and $\hat \Fc$ are generated by the coordinate mapping and are thus included in the respective Borel $\sigma$-field for $L^p$-convergence; see, for example, \cite[Example 1.3 and Theorem 12.5]{Bi}. On the other hand, convergence in the uniform topology (resp.~Skorokhod topology) implies $L^p$-convergence on $\Omega$ (resp.~$\hat \Omega$) and so we have equality of the Borel $\sigma$-fields.
\end{remark}

We start with an auxiliary result.

\begin{proposition}\label{prop:stability_optimiser}
	Let $p\ge 1$. 
	Suppose that there exist unique strong solutions to the SDEs \eqref{eq:sde-path-dep} and write $\mu, \nu$ for their respective laws.
	Let $X^n,\bar X^n \colon \Omega\to\hat\Omega$, $n\in\N$, be measurable and suppose that for any $\rho$-correlated Wiener process $(W,\bar W)$,
	\begin{align}\label{eq:stability_cond_wp}
	\W_p(\pi^{\rho,n},\pi^\rho)\xrightarrow{n\to\infty}0,
\end{align}
	where $\pi^{\rho,n}=\Law(X^n\circ W,\bar X^n\circ \bar W)$ and $\pi^\rho=\Law(X,\bar X)$ with $X,\bar X$ solving \eqref{eq:sde-path-dep} with respect to $(W,\bar W)$. 
	 Let $c \colon \mathbb R\times \mathbb R \to\mathbb R$ be continuous and satisfy \eqref{eq:p_growth_c} for some $K>0$ and suppose that there exists some $\hat\rho\colon [0, 1] \times \Omega \times \Omega \to [-1, 1]$ such that for each $n\in\N$, $\pi^{\hat\rho,n}$ attains
\[
\inf_{\pi\in\widetilde\cpl(\mu^n,\nu^n)}\iint_0^1c\!\left(\omega_t,\bar{\omega}_t\right)\dt\di\pi,
\]
where $\widetilde\cpl(\mu^n,\nu^n)$ denotes all couplings of the form $\pi^{\rho,n}$ for some correlation process $\rho$ (writing $\mu^n,\nu^n$ for their marginals). 
	Then
\begin{equation}\label{eq:convergence_general}
\lim_{n\to\infty}\;
\inf_{\pi\in\widetilde\cpl(\mu^n,\nu^n)}\iint_0^1c\!\left(\omega_t,\bar{\omega}_t\right)\dt\di\pi\;
=
\inf_{\pi\in\cplbc(\mu,\nu)}\iint_0^1c\!\left(\omega_t,\bar{\omega}_t\right)\dt\di\pi;
\end{equation}
moreover, the right hand side is attained by $\pi^{\hat\rho}$.
\end{proposition}

\begin{proof}
	By assumption, there exists some $\hat\rho$-correlated Wiener process such that, for each $n\in\N$, and for any admissible correlation process $\rho$,
	\begin{align}\label{eq:revision_proof_stability}
		\iint_0^1c(\omega_t,\bar\omega_t)\dt\di\pi^{\hat\rho,n}
		\le 	\iint_0^1c(\omega_t,\bar\omega_t)\dt\di\pi^{\rho,n}.
	\end{align}
	
	We now argue that the functional $(x(\cdot),y(\cdot))\mapsto\int_0^1c(x(t),y(t))\dt$ is continuous with at most polynomial growth of order $p$ on with respect to the $L^p$-metric on $\hat \Omega \times \hat \Omega$.
	The growth claim is immediate since
	\begin{align}
		\!\left|\int_0^1c(x(t),y(t))\dt\right|\!
		\le \int_0^1K(1+|x(t)|^p+|y(t)|^p)\dt
		=K(1+\|x\|_{L^p}^p+\|y\|_{L^p}^p).
	\end{align}
	Suppose now that $\|x^n-x\|_{L^p}+\|y^n-y\|_{L^p}\to 0$. The sequences $\{|x^n(\cdot)|^p\}_n$ and $\{|y^n(\cdot)|^p\}_n$ are then uniformly integrable w.r.t. Lebesgue measure on $[0,1]$. By the growth assumption on $c$, the sequence $\{|c(x^n(\cdot),y^n(\cdot))|\}_n$ is similarly uniformly integrable. 
	Define $a_n \coloneqq \int_0^1c(x^n(t),y^n(t))\dt$. For any subsequence $\{a_{n_k}\}_k$ we can find a sub-subsequence $\{a_{n_{k_j}}\}_j$ such that $c(x^{n_{k_j}}(\cdot),y^{n_{k_j}}(\cdot))\to c(x(\cdot),y(\cdot))$ almost surely, as $c$ is continuous. Then by uniform integrability we infer that $a_{n_{k_j}}\to a \coloneqq \int_0^1c(x(t),y(t))\dt$. As the limit does not depend on the subsequence, we conclude that $a_n\to a$.

Equipping $\hat\Omega$ with the topology induced by the $L^p$-metric, it follows from \Cref{rem:borel-sets} that the product space $\hat \Omega \times \hat \Omega$ is a separable Radon space. Given this property, the continuity and polynomial growth shown above, and the $\W_p$-convergence \eqref{eq:stability_cond_wp}, we can apply \cite[Lemma~5.1.7 and Proposition~7.1.5]{AGS} to take the limit in \eqref{eq:revision_proof_stability} as $n\to\infty$, and we deduce that, for any admissible correlation process $\rho$, 
\begin{align}
	\iint_0^1c(\omega_t,\bar\omega_t)\dt\di\pi^{\hat\rho}
	\le 	\iint_0^1c(\omega_t,\bar\omega_t)\dt\di\pi^{\rho}.
\end{align}

By \Cref{prop:bion-nadal-talay-bc}, $\pi^{\hat\rho}$ thus attains the infimum on the right-hand side of \eqref{eq:convergence_general}. 
Once again using the convergence of $\pi^{\hat\rho,n}$ to $\pi^{\hat\rho}$ in $\W_p$, together with the optimality of these couplings for their respective transport problems, we conclude that the value of the problem also converges.
\end{proof}

We now make use of the above result to obtain two stability results where the assumptions are placed directly on the marginals; they will be crucial for our upcoming analysis. The first one allows for c{\`a}dl{\`a}g approximations. 

\begin{corollary}\label{rem:suff_cond_stab_lp}
	Suppose that there exist unique strong solutions $X,\bar X$ to the SDEs \eqref{eq:sde-path-dep}. Let $p\ge 1$ and $X^n,\bar X^n \colon \Omega\to\hat\Omega$, $n\in\N$, be measurable and such that $X^n\circ W$ (resp. $\bar X^n\circ \bar W$) converges in $L^p$ to $X$ (resp. $\bar X$). Let $c \colon \mathbb R\times \mathbb R \to\mathbb R$ be continuous and satisfy \eqref{eq:p_growth_c}, and suppose that there exists some $\hat\rho$ as in \Cref{prop:stability_optimiser}. Then condition \eqref{eq:stability_cond_wp} is satisfied and the conclusions of \Cref{prop:stability_optimiser} hold true. 
\end{corollary}
 	
\begin{proof}
	For any $\rho$-correlated Wiener process $(W,\bar W)$,  let $(X,\bar X)$ be the solution of the SDE \eqref{eq:sde-path-dep} with respect to $(W,\bar W)$ and define $(X^n,\bar X^n)=(X^n\circ W,\bar X^n\circ \bar W)$. Then $\mathrm{Law}((X^n,\bar X^n),(X,\bar X))\in\cpl(\pi^{\rho,n},\pi^\rho)$. Moreover, by the assumption of $L^p$-convergence,
    \begin{align}
        \E\!\left[\int_0^1|(X^n_t,\bar X^n_t)-(X_t,\bar X_t)|^p\dt\right]^\frac{1}{p}
        \le \E\!\left[\int_0^1|X^n_t-X_t|^p\dt\right]^\frac{1}{p}+\E\!\left[\int_0^1|\bar X^n_t-\bar X_t|^p\dt\right]^\frac{1}{p}
        \xrightarrow{n\to\infty}0,
    \end{align} 
    which verifies \eqref{eq:stability_cond_wp} and thus completes the proof. 
\end{proof}

Our second stability result enables us to approximate the adapted Wasserstein distance between laws of SDEs by approximating their coefficients. For this result we return to the set-up of continuous paths in $\Omega$; we write $\|\omega\|_\infty \coloneqq  \sup_{s\in[0,1]}|\omega_s|$, $\omega\in \Omega$, for the sup-norm and work under the following assumption.

\begin{assumption}\label{ass:stability}
Suppose that $x_0,\bar x_0 \in \R$ and $b, \bar b\colon [0, 1] \times \Omega \to \R$, $\sigma, \bar \sigma\colon [0, 1] \times \Omega \to \R_+$ satisfy the following:
	\begin{enumerate}[label = (\roman*)]
		\item $b,\bar b,\sigma,\bar\sigma$ are progressively measurable;
		\item for each $t\in [0, 1]$, the functions $b(t,\cdot),\bar b(t,\cdot),\sigma(t,\cdot),\bar\sigma(t,\cdot)$ are continuous w.r.t.\ $\|\cdot\|_\infty$;
		\item there exists $K>0$ such that, for all $t\in [0, 1]$, $\omega \in \Omega$,
			\begin{equation}\label{eq:linear-growth}
				|b(t,\omega)|\vee|\bar b(t,\omega)|\vee|\sigma(t,\omega)|\vee|\bar\sigma(t,\omega)|\leq K(1+\|\omega\|_\infty);
			\end{equation}
		\item there exist unique strong solutions of the SDEs \eqref{eq:sde-path-dep}.	
	\end{enumerate}
\end{assumption}

\begin{corollary}\label{cor:synchronous-stability} 
Let $(x_0,\bar x_0, b, \bar{b}, \sigma, \bar{\sigma})$ satisfy \Cref{ass:stability} and write $\mu, \nu$ for the laws of $X,\bar X$.
For $n\in\N$, consider $(x^n_0, \bar x^n_0,b^n,\bar b^n,\sigma^n,\bar\sigma^n)$ satisfying \Cref{ass:stability}.(i) and (iii), with a uniform slope constant $K$ in \eqref{eq:linear-growth}, and such that strong existence holds for \eqref{eq:sde-path-dep}; write $X^n,\bar X^n$ for a pair of solutions and $\mu^n, \nu^n$ for their laws.
Let $c \colon \mathbb R\times \mathbb R \to\mathbb R$ be continuous and satisfy \eqref{eq:p_growth_c}, and suppose that there exists some $\hat\rho$ as in \Cref{prop:stability_optimiser}.

Suppose that, as $n \to \infty$, $(x^n_0,\bar x^n_0) \to (x_0,\bar x_0)$ and the following convergence holds:
\begin{align}\label{eq:unifconvassump}
\|\omega^n - \omega \|_\infty \to 0 \implies (b^n,\bar b^n,\sigma^n,\bar\sigma^n)(t,\omega^n)\to (b,\bar b,\sigma,\bar\sigma)(t,\omega), \text{ for each }t \in [0, 1].
\end{align}
Then condition \eqref{eq:stability_cond_wp} is satisfied and the conclusions of \Cref{prop:stability_optimiser} hold true.
\end{corollary}

\begin{proof}
	Let $\rho$ be an admissible correlation process. Write $\pi^{\rho, n}$ for the joint law of $(X^n,\bar X^n)$, for $n \in \N$, and $\pi^\rho$ for the joint law of $(X,\bar X)$, when the corresponding SDEs are driven by a $\rho$-correlated Wiener process. By \Cref{prop:stab_abstract}, we have convergence of $\pi^{\rho,n}$ to $\pi^\rho$ in the $p$-Wasserstein distance on $\Pc(\Omega \times \Omega)$ with respect to the sup-norm. 
	Embedding $\Pc(\Omega\times\Omega)$ into $\Pc(\hat\Omega\times\hat\Omega)$, we thus have convergence of $\pi^{\rho,n}$ to $\pi^\rho$ in the $p$-Wasserstein distance on $\Pc(\hat\Omega\times\hat\Omega)$ with respect to the $L^p$-norm (as defined in \eqref{eq:distance_Wp_Lp}), which verifies \eqref{eq:stability_cond_wp} and thus completes the proof. 
\end{proof}

\begin{remark}\label{rem:corr_uniqueness}
Whenever $(W,\bar W)$ is a $\rho$-correlated Wiener process defined on some stochastic basis, under \Cref{ass:stability} (iv), one can uniquely construct a strong solution $(X,\bar X)$ of the system \eqref{eq:sde-path-dep} driven by $(W,\bar W)$ on the same stochastic basis.
We also note that \Cref{ass:stability} (iv) can be weakened to pathwise uniqueness only. Indeed, weak existence is guaranteed already by \Cref{ass:stability} (i)--(iii) and a classical result of Skorokhod (e.g.\ adapting the proof of \cite[Theorem 21.9]{Kall02} to coefficients with linear growth) and so the Yamada-Watanabe criterion \cite[Lemma 21.17]{Kall02} applies.
	Note finally that continuity of $b, \bar b, \sigma, \bar \sigma$ is implied by the convergence \eqref{eq:unifconvassump} and so we could drop \Cref{ass:stability} (ii) from our assumptions. We keep this assumption, however, to make it transparent that the stability result of \Cref{cor:synchronous-stability} cannot be applied to coefficients with discontinuities. 
\end{remark}

\subsection{Approximation of SDEs in adapted Wasserstein distance}

We next make use of \Cref{prop:bion-nadal-talay-bc} to prove a more general version of \Cref{thm:convergence} on the approximation of the laws of SDEs in adapted Wasserstein distance. In \Cref{prop:aw-conv-kr-sync}, we will apply this result to the convergence of the Knothe--Rosenblatt rearrangement to the synchronous coupling.
For $p \ge 1$, and $\pi,\pi^\prime\in \Pc(\Omega \times \Omega)$ with finite $p$\textsuperscript{th} moment, we define the adapted $p$-Wasserstein distance between $\pi$ and $\pi^\prime$ analogously to \eqref{eq:adapted_wasserstein}; specifically, defining the set of bicausal couplings, $\cplbc(\pi,\pi^\prime)$, analogously to \Cref{def:bicausal_couplings} when $\Omega\times\Omega$ is equipped with the product filtration, $\AW_p(\pi,\pi^\prime)$ is given by \eqref{eq:distance_Wp_Lp} when replacing the set $\cpl(\pi,\pi')$ by $\cplbc(\pi,\pi')$.

\begin{theorem}\label{thm:convergence-general}
	Suppose that for some $p \ge 1$, $b, \bar b\colon C([0, 1], \R) \to \R$, $\sigma, \bar \sigma\colon C([0, 1], \R) \to \R_+$ there exist $p$-integrable unique strong solutions $(X, \bar X)$ of \eqref{eq:sde-path-dep}. Suppose also that, for all $h > 0$, and $b^h, \bar b^h\colon C([0, 1], \R) \to \R$, $\sigma^h, \bar \sigma^h\colon C([0, 1], \R) \to \R$ there exist $p$-integrable unique strong solutions $(X^h, \bar X^h)$ of \eqref{eq:sde-path-dep}. Moreover, suppose that $X^h \to X$ and $\bar X^h \to \bar X$ in $L^p$.
	
	Then, for any bicausal coupling $\pi \in \cplba(\Law(X), \Law(\bar X))$, there exists a probability space supporting a correlated Wiener process $(W, \bar{W})$ such that $\pi$ is equal to the joint law of $(X, \bar X)$ driven by $(W, \bar W)$ and, for $\pi^h$ equal to the joint law of $(X^h, \bar X^h)$ driven by $(W, \bar W)$, we have
$
		\lim_{h \to 0}\AW_p(\pi^h, \pi) = 0
$.
\end{theorem}

\begin{remark}\label{rem:convergence-EM-aw}
	Under Lipschitz conditions on the coefficients of the SDE \eqref{eq:sde}, the Euler--Maruyama scheme \eqref{eq:approx} converges to the unique solution of \eqref{eq:sde} in $L^p$, for all $p \ge 1$ (see, e.g.\ \cite{KlPl92}), and so \Cref{thm:convergence-general} implies the result of \Cref{thm:convergence}. Under the same conditions, the monotone Euler--Maruyama scheme that we develop below in \Cref{def:monotone-em} also converges to the solution of \eqref{eq:sde} in $L^p$, for all $p \geq 1$; see \Cref{lem:convergence-truncated,prop:aw-conv-kr-sync}.
\end{remark}

\begin{proof}[Proof of \Cref{thm:convergence-general}]
	Take $\pi \in \cplba(\Law(X), \Law(\bar X))$. By \Cref{prop:bion-nadal-talay-bc}, there exists a correlated Wiener process $(W, \bar{W})$ such that $\pi$ is the joint law of $(X, \bar{X})$ driven by $(W, \bar W)$. Take $(X^h, \bar X^h)$ driven by $(W, \bar W)$. Then $\Law((X^{h}, \bar{X}^{h}), (X, \bar{X})) \in \cplba(\pi^h, \pi)$. Indeed, for $t \in [0, 1]$, and any bounded measurable $f\colon C([0, t], \R) \times C([0, t], \R) \to \R$,
	\begin{equation}
		\begin{split}
			&\E \!\left[f((X^h_s, \bar X^h_s)_{s \in [0, t]}) \mid \; \F^X_1 \otimes \F^{\bar X}_1 \right]\\
			& \qquad  = \E \!\left[ \E\!\left[f((X^h_s, \bar X^h_s)_{s \in [0, t]}) \mid \; (\F^X_t \otimes \F^{\bar X}_t) \vee \sigma\{(W_u - W_t, \bar W_u - \bar W_t) \colon u \in (t, 1]\}\right] \mid \; \F^X_1 \otimes \F^{\bar X}_1 \right]\\
			& \qquad = \E \!\left[ \E\!\left[f((X^h_s, \bar X^h_s)_{s \in [0, t]}) \mid \; \F^X_t \otimes \F^{\bar X}_t \right] \mid \; \F^X_1 \otimes \F^{\bar X}_1 \right]\\
			& \qquad = \E\!\left[f((X^h_s, \bar X^h_s)_{s \in [0, t]}) \mid \; \F^X_t \otimes \F^{\bar X}_t \right],
		\end{split}	
	\end{equation}
	where the second equality follows from the fact that $\sigma\{(W_u - W_t, \bar W_u - \bar W_t) \colon u \in (t, 1]\}$ is independent of $(X^h_s, \bar X^h_s)_{s \in [0, t]}$ and $\F^X_t \otimes \F^{\bar X}_t$. This implies causality in one direction. Since the roles of $(X^h, \bar X^h)$ and $(X, \bar X)$ are symmetric in this calculation, we have bicausality.
	
	Finally, using the convergence in $L^p$ combined with the same arguments as used in \Cref{rem:suff_cond_stab_lp}, we conclude that $\AW_p(\Law(X^{h}, \bar{X}^{h}), \pi)$ converges to zero, as $h\to 0$. 
\end{proof}

\section{The synchronous coupling: Properties and optimality}\label{sec:proofs}

We now return to the setting of SDEs with time-homogeneous Markovian coefficients. Specifically, we consider functions $b\colon \R \to \R$ and $\sigma \colon \R \to \R_+$ such that there exists a unique strong solution of the SDE \eqref{eq:sde}; we write $\mu^{b, \sigma}$ for its law. Without loss of generality, we suppose that all SDEs start from the same initial value $x_0$ and so we omit $x_0$ from any notation. 

In order to define the synchronous coupling between any two such laws, say $\mu=\mu^{b,\sigma}$ and $\nu=\mu^{\bar b,\bar\sigma}$, consider the system of SDEs, defined for $t \in [0, 1]$,
\begin{equation}\label{eq:sdes}
	\begin{split}
		\D X_t & = b(X_t)\D t + \sigma(X_t) \D W_t; \quad X_0 = x_0,\\
		\D \bar{X}_t & = \bar{b}(\bar{X}_t)\D t + \bar{\sigma}(\bar{X}_t)\D \bar{W}_t; \quad \bar{X}_0 = x_0.
	\end{split}
\end{equation}
We then define the \emph{synchronous coupling} $\pi^\sync_{\mu, \nu} \in \cplbc(\mu, \nu)$ as follows:\footnote{If one drops the assumption that $\sigma, \bar \sigma$ are both positive, one could recover the results of this paper by redefining the synchronous coupling to be the one induced by $(W, \bar W)$ with correlation $\sign(\sigma \bar \sigma)$, where $\sign(x)=1$ if $x\geq 0$ and $-1$ otherwise. This determines the synchronous coupling uniquely, even in situations where $\sigma\bar\sigma$ can become zero.} Set $\bar{W} = W$ and let $(X^\sync, \bar{X}^\sync)$ be the strong solution of \eqref{eq:sdes} driven by $(W, W)$. Then the \emph{synchronous coupling} is defined as the joint law
\begin{equation}
    \pi^\sync_{\mu, \nu}\coloneqq  \Law(X^\sync, \bar{X}^\sync).
\end{equation}
We will also refer to the coupling obtained via the above procedure with $\bar W=-W$ as the \emph{anti-synchronous coupling}, $\pi^{\async}_{\mu,\nu}$. 

The main aim of this section is to establish the optimality of the synchronous coupling for the adapted Wasserstein distance between laws of SDEs for which pathwise uniqueness holds and whose coefficients are continuous and have linear growth. 
We also establish this result for a particular class of coefficients which allows for discontinuities in the drift coefficient.
As part of our analysis, we also provide an optimality result for the Knothe--Rosenblatt rearrangement and establish the link to the synchronous coupling. 

\subsection{The Knothe--Rosenblatt rearrangement}\label{sec:kr}

The Knothe--Rosenblatt rearrangement (also known as the Knothe--Rosenblatt coupling or quantile transformation) was introduced independently by Rosenblatt \cite{Ro52} and Knothe \cite{Kn57}. This coupling can be seen as an extension of the monotone rearrangement between marginal laws on $\R$ to the case of coupling laws of $\R$-valued discrete-time processes. We illustrate the Knothe--Rosenblatt rearrangement in \Cref{fig:kr}. The aim here is to investigate its optimality properties. To this end, we first introduce some notation.

	For a probability measure $\mu$ on $\R$, the \emph{cumulative distribution function} $F_\mu\colon \R \to [0, 1]$, is given by 
	\begin{align}
		F_\mu(u)= \mu(-\infty, u],
		\quad u\in\R.
	\end{align} 
	The \emph{quantile function} $F_\mu^{-1}\colon [0,1]\to\R$ is defined as its left-continuous inverse; that is, $F_\mu^{-1}(y) = \inf\{u \in \R \colon F_\mu(u) \geq y\}$.
	Given two probability measures $\mu,\nu$ on $\R$, recall that $\mu$ is said to dominate $\nu$ in \emph{first order stochastic dominance} if $F_\mu(u)\le F_\nu(u)$, for all $u \in \R$. 

	Given a probability measure $\mu$ on $\R^n$, $n \in \N$, let $\mu_1$ be its marginal distribution onto the first component and write $\mu_{x_1, \dotsc, x_k}$, $x_1, \dotsc, x_k\in\R$, $k=1,\dots,n-1$, for the one-dimensional conditional distribution in the $(k+1)$\textsuperscript{th} coordinate given the first $k$ coordinates; that is
	\begin{align}\label{eq:disintegration_repr}
	\mu(\dx_1,\dots,\dx_n)=\mu_1(\dx_1)\mu_{x_1}(\dx_2)\dots\mu_{x_1,\dots,x_{n-1}}(\dx_n).
	\end{align}
	If $\mu$ defines a Markov process, then the transition kernels $\mu_{x_1, \dotsc, x_k}$, $k=1,\dots,n-1$, will only depend on $x_k$ but we keep writing out the full tuple for notational reasons. 
	
	Two functions $f, g\colon \R \to \R$ are called \emph{co-monotone} in each of the following three cases: $f$ and $g$ are both increasing, $f$ and $g$ are both decreasing, or one of $f$ and $g$ is constant and the other is arbitrary. Here, and throughout the paper, we use the terms increasing and decreasing in a weak sense; we do not require strict monotonicity.
	
	We are now ready to define the Knothe--Rosenblatt rearrangement:
	
\begin{definition}[Knothe--Rosenblatt rearrangement]\label{def:kr}
	Given probability measures $\mu, \nu$ on $\R^n$, let $U_1, \dotsc, U_n$ be independent uniform random variables on $[0, 1]$, define $X_1 = F_{\mu_1}^{-1}(U_1)$, $Y_1 = F_{\nu_1}^{-1}(U_1)$, and, for $k = 2, \dotsc, n$, define inductively the random variables
	\begin{equation}
		X_k = F^{-1}_{\mu_{X_1, \dotsc, X_{k - 1}}}(U_k), \quad Y_k = F^{-1}_{\nu_{Y_1, \dotsc, Y_{k - 1}}}(U_k).
	\end{equation}
	The \emph{Knothe--Rosenblatt rearrangement} between the marginals $\mu$ and $\nu$ is then given by 
	\[
	\pi^\kr_{\mu, \nu} \coloneqq \Law(X_1, \dotsc, X_n, Y_1, \dotsc, Y_n).
	\]
	For $n = 1$, $\pi^\kr_{\mu, \nu} = \Law(F_{\mu_1}^{-1}(U_1), F_{\nu_1}^{-1}(U_1))$ is known as the \emph{monotone rearrangement}.
\end{definition}

\begin{remark}\label{rem:monge-kr}
	If $\mu$ is absolutely continuous with respect to the Lebesgue measure, then $\pi^\kr_{\mu, \nu}$ is induced by the Monge map $(x_1, \dotsc, x_n) \mapsto T(x_1, \dotsc, x_n) = (T^1(x_1), T^2(x_2; x_1), \dotsc, T^n(x_n; x_1 \dotsc, x_{n-1}))$ given by $T^1(x_1) = F_{\nu_1}^{-1} \circ F_{\mu_1}(x_1)$ and, for $k = 2, \dotsc, n$,
	\begin{equation}
		T^k(x_k; x_1, \dotsc, x_{k-1}) = F^{-1}_{\nu_{T^1(x_1), \dotsc, T^{k - 1}(x_{k - 1}; x_1, \dotsc, x_{k - 2})}}\circ F_{\mu_{x_1, \dotsc, x_{k - 1}}}(x_k).
	\end{equation}
\end{remark}

\begin{remark}\label{rem:kr-increasing-map}
	The following condition is equivalent to the condition in \Cref{def:kr}: For independent uniform random variables $U_1, \dotsc, U_n$ on $[0, 1]$, we have the representation
	\begin{equation}
		\begin{split}
			X & = (T_1(U_1), T_2(U_2; X_1), \dotsc, T_n(U_n; X_1, \dotsc, X_{n - 1})),\\
			Y & = (S_1(U_1), S_2(U_2; Y_1), \dotsc, S_n(U_n; Y_1, \dotsc, Y_{n-1})),
		\end{split}
	\end{equation}
	where the functions $T_i, S_i$ are co-monotone in their first argument, for all $i = 1, \dotsc, n$. By composing each $T_i$ and $S_i$ in their first argument with suitable increasing functions, we could replace each $U_i$ by an independent normal random variable $Z_i\sim\mathcal N(0,1/n)$, for $i= 1, \dotsc, n$. This shows more clearly the resemblance between the synchronous coupling in continuous time and the Knothe--Rosenblatt rearrangement in discrete time. Much of what we do in this paper can be interpreted as justifying a convergence of sorts of the latter to the former when $n\to\infty$.
\end{remark}

We now consider the optimality properties of the Knothe--Rosenblatt rearrangement.
Suppose that the cost function $c \colon \R \times \R \to \R$ is \emph{submodular}\footnote{Submodular functions are also referred to as quasi-monotone or $L$-subadditive.}; i.e.
\begin{equation}\label{eq:L_superadditivity}
	c(x,y)+c(\bar x,\bar y)-c(x,\bar y)- c(\bar x,y)\le 0,
	\quad \textrm{for all $x\le \bar x$, $y\le \bar y$}.
\end{equation}
For probability measures $\mu, \nu$ on $\R^n$, $n \in \N$, we define the set $\cplba(\mu, \nu)$ of bicausal couplings analogously to \Cref{def:bicausal_couplings} and, for $p \geq 1$, we write $\mu \in \Pc_p(\R^n)$ if $\int_{\R^n}\sum_{k = 1}^{n} |x_k|^p \mu(\D x) < \infty$.

We relate the optimality of the Knothe--Rosenblatt rearrangement for a class of discrete-time bicausal transport problems to the concept of stochastic monotonicity.
For Markov processes, the notion of stochastic monotonicity goes back to \citet{Da68}. The following definition generalises this concept to arbitrary processes.
	
\begin{definition}[stochastic co-monotonicity]\label{def:stochastic_mon}
	A probability measure $\mu$ on $\R^n$ (or a stochastic process with law $\mu$) is \emph{stochastically increasing} (resp. decreasing) if, for $k = 1, \dotsc, n -1$, the map $(x_1,\dots,x_k)\mapsto\mu_{x_1,\dotsc,x_k}$ is increasing (resp.~decreasing) in first order stochastic dominance on $\Pc(\R)$ with respect to the product order on $\R^k$.\footnote{I.e.~$\mu_{\bar x_1,\dotsc,\bar x_k}$ dominates $\mu_{x_1,\dotsc, x_k}$ in first order stochastic dominance for any $(x_1, \dotsc, x_k), (\bar x_1, \dotsc, \bar x_k) \in \R^k$ with $x_i \leq \bar x_i$, for all $i = 1, \dots, k$.}
	
	We say that two probability measures on $\R^n$ are \emph{stochastically co-monotone} if they are both stochastically increasing, both stochastically decreasing, or one is both stochastically increasing and decreasing and the other is arbitrary. 
\end{definition}

The following optimality result forms the basis for our proof of optimality of the synchronous coupling in continuous time.
For the particular case of a two-period problem and a cost function of the form $c(x,y)=f(x-y)$, $x,y\in\R$, for some convex function $f$, the result was established in \cite[Corollary~2]{Ru85}; for the same type of cost functions and multi-period Markov processes, the result was established in \cite[Proposition~5.3]{BaBeLiZa17}; see, however, \Cref{rem:stoch-mono} below.

\begin{proposition}[optimality of Knothe--Rosenblatt]\label{prop:kr-discrete-optimality}
	For $p\ge 1$ and $n\in\mathbb{N}$, let $\mu,\nu \in \Pc_p(\R^n)$ be stochastically co-monotone.
	 Suppose that, for some $K > 0$, $c \colon \R\times\R\to\R$ satisfies 	\eqref{eq:p_growth_c} and \eqref{eq:L_superadditivity}.
	Then the Knothe--Rosenblatt rearrangement between $\mu$ and $\nu$
	is optimal for the bicausal transport problem
	\begin{align}\label{eq:optimisation_KR_problem}
		\inf_{\pi \in \cplbc(\mu, \nu)}\int \sum_{k=1}^nc(x_k,y_k)\pi(\dx,\dy).
	\end{align}

    Suppose additionally that there exists a strictly convex function $f \colon \R \to \R$ and a function $\tilde c \colon \R \times \R \to \R$ satisfying \eqref{eq:L_superadditivity} such that $c(x, y) = f(x - y) + \tilde c(x, y)$ for all $x, y \in \R$. Then the Knothe--Rosenblatt rearrangement is the unique optimiser of \eqref{eq:optimisation_KR_problem}.
\end{proposition}

Before proceeding to the proof, two remarks concerning classical optimal transport on the line are necessary.

\begin{remark}\label{rem:monotone-optimality}
	Let $p \geq 1$, $\alpha, \beta \in \Pc_p(\R)$, and suppose that, for some $K > 0$, $c \colon \R \times \R \to \R$ satisfies \eqref{eq:p_growth_c} and \eqref{eq:L_superadditivity}. Then, by \citet[Corollary 3]{Ru83}, the monotone rearrangement defined in \Cref{def:kr} attains the infimum in the optimal transport problem
		\begin{align}
			\inf_{\pi \in \cpl(\alpha, \beta)}\int c(x, y) \pi(\D x, \D y).
		\end{align}
		This characterisation of optimality in terms of monotonicity is inherently one-dimensional. To prove optimality of the Knothe--Rosenblatt rearrangement, we iterate this result over multiple time steps. Since we deduce the optimality of the synchronous coupling from the discrete-time result by a limiting argument, the proof of our continuous-time result is also specific to one-dimensional processes; see \Cref{prop:bnt}.
\end{remark}

\begin{remark}\label{rem:monotone-uniqueness}
    Let $f \colon \R \to \R$ be strictly convex and $c_f(x, y) := f(x - y)$, for any $x, y \in \R$. Let $p \geq 1$ and $\alpha, \beta \in \Pc_p(\R)$. Then, by \Cref{rem:monotone-optimality}, the monotone rearrangement defined in \Cref{def:kr} attains the infimum
    \begin{align}
		\inf_{\pi \in \cpl(\alpha, \beta)}\int f(x - y) \pi(\D x, \D y).
	\end{align}
    By \cite[Theorem 1]{BeGoMaSc08}, any coupling that attains this infimum is $c_f$-cyclically monotone. One can verify that, for $f$ strictly convex on $\R$, the only $c_f$-cyclically monotone coupling is the monotone rearrangement, and thus this is the \emph{unique} optimiser.
    
    Now let $c \colon \R \times \R \to \R$ take the form prescribed in the uniqueness part of \Cref{prop:kr-discrete-optimality}; i.e.\ $c=c_f+\tilde c$ with $c_f$ as defined above and $\tilde c$ submodular.
    By \Cref{rem:monotone-optimality}, the monotone rearrangement also attains the infima
    	\begin{align}
			\inf_{\pi \in \cpl(\alpha, \beta)}\int c(x, y) \pi(\D x, \D y) \quad \text{and} \quad \inf_{\pi \in \cpl(\alpha, \beta)}\int \tilde c(x, y) \pi(\D x, \D y).
		\end{align}
    Therefore
    \begin{align}
		\inf_{\pi \in \cpl(\alpha, \beta)}\int c(x, y) \pi(\D x, \D y) =  \inf_{\pi \in \cpl(\alpha, \beta)}\int f(x - y) \pi(\D x, \D y) + \inf_{\pi \in \cpl(\alpha, \beta)}\int \tilde c(x, y) \pi(\D x, \D y),
	\end{align}
    and uniqueness of the optimiser for the problem on the left-hand side follows.
\end{remark}

\begin{proof}[Proof of \Cref{prop:kr-discrete-optimality}]
	Similarly to \eqref{eq:disintegration_repr}, we identify each measure $\pi$ on $\R^n\times\R^n$ with its associated sequence of consecutive disintegration kernels
	\begin{align}\label{eq:disintegration}
\pi(\dx_1,\dots,\dx_n,\dy_1,\dots,\dy_n)
	=\pi_1(\dx_1,\dy_1)\pi_{x_1,y_1}(\dx_2,\dy_2)\dots
	\pi_{x_{1:n-1},y_{1:n-1}}
	(\dx_n,\dy_n),
		\end{align}
	where we use the shorthand notation $x_{1:k}$ for $x_1,\dots,x_k$, and analogously for the $y$-component. 
	By \cite[Proposition~5.1]{BaBeLiZa17}, $\pi\in\mathrm{Cpl_{bc}}(\mu,\nu)$ if and only if $\pi_1\in\mathrm{Cpl}(\mu_1,\nu_1)$ and $\pi_{x_{1:k},y_{1:k}}\in\mathrm{Cpl}(\mu_{x_{1:k}},\nu_{y_{1:k}})$, for $k=1,\dots,n-1$. 
	This allows us to solve the optimisation problem \eqref{eq:optimisation_KR_problem} by backward induction. Define $v_n(x_{1:n},y_{1:n})\equiv 0$ and, for $k=2,\dots,n$, define $v_{k-1}$ inductively by
\begin{align}\label{eq:KR_DPP}
v_{k-1}\!\left(x_{1:k-1},y_{1:k-1}\right)
=
\inf_{\pi\in\mathrm{Cpl}(\mu_{x_{1:k-1}},\nu_{y_{1:k-1}})}\int c(x_k,y_k)+v_{k}(x_{1:k},y_{1:k})\;\pi(\dx_k,\dy_k).
\end{align}
Then, by \cite[Proposition~5.2]{BaBeLiZa17}, the value of problem \eqref{eq:optimisation_KR_problem} is given by 
$v_0=
\inf_{\pi\in\mathrm{Cpl}(\mu_1,\nu_1)}\int c(x_1,y_1)+v_1(x_1,y_1)\;\pi(\dx_1,\dy_1)$,
and \eqref{eq:optimisation_KR_problem} admits a minimiser $\pi$, which is obtained from the sequence of locally optimal disintegration kernels via \eqref{eq:disintegration}. 

For $k=n$, the infimum in \eqref{eq:KR_DPP} is attained by the monotone rearrangement between $\mu_{x_1:x_{n-1}}$ and $\nu_{y_1:y_{n-1}}$, since $c$ satisfies \eqref{eq:L_superadditivity}; see \Cref{rem:monotone-optimality}. Therefore
\begin{align}
v_{n-1}(x_{1:n-1},y_{1:n-1})
=
\int_0^1 c\!\left(F^{-1}_{\mu_{x_1:x_{n-1}}}(u),F^{-1}_{\nu_{y_1:y_{n-1}}}(u)\right)\di u.
\end{align}

For $k=n-1$, the assumption of stochastic monotonicity gives that $x_{n-1}\mapsto F^{-1}_{\mu_{x_1:x_{n-1}}}(u)$ and $y_{n-1}\mapsto F^{-1}_{\nu_{y_1:y_{n-1}}}(u)$ are co-monotone for any $x_{1:n-2}$, $y_{1:n-2}$ and $u \in [0, 1]$. We also observe that \eqref{eq:L_superadditivity} is a linear constraint and that, if $c$ satisfies \eqref{eq:L_superadditivity} and $f, g \colon \R \to \R$ are co-monotone, then $(x,y)\mapsto c(f(x),g(y))$ also satisfies \eqref{eq:L_superadditivity}. Thus $(x_{n-1},y_{n-1})\mapsto c(x_{n-1},y_{n-1})+v_{n-1}(x_{1:n-1},y_{1:n-1})$ satisfies \eqref{eq:L_superadditivity} for any $x_{1:n-2}$, $y_{1:n-2}$.
As a consequence, the monotone rearrangement again attains the infimum on the right-hand side of \eqref{eq:KR_DPP}, and so
\begin{align}
v_{n-2} & (x_{1:n-2},y_{1:n-2})
=
\int_0^1 \bigg\{c\!\left(F^{-1}_{\mu_{x_1:x_{n-2}}}(s),F^{-1}_{\nu_{y_1:y_{n-2}}}(s)\right)\\
+&
\int_0^1 c\!\left(F^{-1}_{\mu_{x_1:x_{n-2},x_{n-1}}}(u),F^{-1}_{\nu_{y_1:y_{n-2},y_{n-1}}}(u)\right)\di u\Big|_{(x_{n-1},y_{n-1})=\!\left(F^{-1}_{\mu_{x_1:x_{n-2}}}(s),F^{-1}_{\nu_{y_1:y_{n-2}}}(s)\right)}\bigg\}\di s.
\end{align}

For $k=n-2$, the assumption of stochastic monotonicity gives that $x_{n-2}\mapsto \mu_{x_1:x_{n-2}}$ and $y_{n-2}\mapsto \nu_{y_1:y_{n-2}}$ as well as $(x_{n-2},x_{n-1})\mapsto \mu_{x_1:x_{n-2},x_{n-1}}$ and $(y_{n-2},y_{n-1})\mapsto \nu_{y_1:y_{n-2},y_{n-1}}$ are co-monotone for any $x_{1:n-3}$, $y_{1:n-3}$. Using once again the fact that \eqref{eq:L_superadditivity} is a linear constraint which is preserved under compositions with co-monotone functions, we have that $(x_{n-2},y_{n-2})\mapsto c(x_{n-2},y_{n-2})+v_{n-2}(x_{1:n-2},y_{1:n-2})$ satisfies \eqref{eq:L_superadditivity}, and so the monotone rearrangement attains the infimum on the right-hand side of \eqref{eq:KR_DPP}. 

By induction, the analogous conclusion holds for $k=1,\dots,n-3$, and we conclude that the Knothe--Rosenblatt rearrangement attains the infimum in \eqref{eq:optimisation_KR_problem}.

To show uniqueness of the optimal coupling, first note that, for any coupling $\pi$ that attains the infimum in \eqref{eq:optimisation_KR_problem}, the disintegration kernel $\pi_{x_{1:k}, y_{1:k}}$ attains the infimum $v_{k - 1}$ in \eqref{eq:KR_DPP}, for each $k = 2, \dotsc, n$, and $\pi_1$ attains the infimum in $v_0$. Thus it suffices to prove uniqueness of the optimiser for each $v_{k - 1}$, $k = 1, \dotsc, n$.

Suppose that $c(x, y) = f(x - y) + \tilde c(x, y)$, for all $x, y \in \R$, where $f$ is strictly convex and $\tilde c$ satisfies \eqref{eq:L_superadditivity}. Then \Cref{rem:monotone-uniqueness} directly implies uniqueness of the optimiser for the infimum in $v_{n - 1}$. Iterating as above, we see that, for each $k = 1, \dotsc, n$, $v_{k -1}$ takes the form required to apply \Cref{rem:monotone-uniqueness} and thus we conclude that the Knothe--Rosenblatt rearrangement is the unique optimiser for \eqref{eq:optimisation_KR_problem}.
\end{proof}

\begin{remark}\label{rem:stoch-mono}
	For Markov processes the assumption of stochastic co-monotonicity reduces to requiring $x_k\mapsto\mu_{x_1,\dots,x_k}$ and $x_k\mapsto\nu_{x_1,\dots,x_k}$ to be co-monotone, $k=1,\dots,n-1$, (since $\mu_{x_1,\dots,x_k}$ only depends on $x_k$). 
	In general, however, this property is not sufficient to ensure optimality of the Knothe--Rosenblatt rearrangement. 
	Indeed, consider the following marginals which satisfy this property but are not stochastically co-monotone in the sense of \Cref{def:stochastic_mon}:
	$\mu=\frac{1}{4}(\delta_{(1,2,7)}+\delta_{(1,0,5)}+\delta_{(-1,0,-5)}+\delta_{(-1,-2,-7)})$
	and 
	$\nu=\frac{1}{4}(\delta_{(1,2,-5)}+\delta_{(1,0,-7)}+\delta_{(-1,0,7)}+\delta_{(-1,-2,5)})$.
	For the cost $c(x,y)=|x-y|$, the Knothe--Rosenblatt rearrangement induces a cost of $12$, while one can find another bicausal coupling that induces a cost of $4$.
	This example reveals that the conclusion of \cite[Proposition~5.3]{BaBeLiZa17} does not hold under the stated assumptions. In \Cref{prop:kr-discrete-optimality} above, we correct and extend this result.
\end{remark}

\begin{remark}
	A well-studied example of a cost function satisfying the above conditions is given by $c(x,y) = |x-y|^p$, $x,y\in \R$, for some $p \geq 1$. In this case, for $\mu, \nu \in \Pc_p(\R^n)$, the infimum in \Cref{prop:kr-discrete-optimality} is the \emph{discrete-time adapted Wasserstein distance}
	\begin{equation}
		\AW_p^p(\mu, \nu) = \inf_{\pi \in \cplba(\mu, \nu)} \int \sum_{k = 1}^n \!\left | x_k - y_k \right |^p \pi(\D x, \D y);
	\end{equation}
	cf.~\cite[Equation~(6)]{BaBaBeEd19b}.
\end{remark}

\begin{remark}\label{rem:kr-dpp-proof}
	The proof of \Cref{prop:kr-discrete-optimality} can be seen as a discrete-time analogue of the proof of optimality of the synchronous coupling between one-dimensional diffusions in \cite{BiTa19}. Indeed, both proofs are based on dynamic programming arguments.
	More precisely, under sufficient regularity assumptions, the crucial observation in the verification argument of \cite{BiTa19} is that the second order cross-derivative of the value function is negative, while the algebraic analogue \eqref{eq:L_superadditivity} of this condition appears in the induction step in the proof of \Cref{prop:kr-discrete-optimality}.
\end{remark}
	
\begin{remark}\label{rem:stoch-mono-fdd}
	Note that the finite-dimensional distributions of any one-dimensional continuous strong Markov process are stochastically increasing, as was shown in \citet[Proposition~5.2]{BePaSc22} by use of a coupling argument originating from \citet{Ho98c}. When such a strong Markov process is the solution of an SDE, our methods give an alternative proof of this fact; see \Cref{rem:stoch-mono-fdd-new}.
\end{remark}

\subsection{A monotone numerical scheme}\label{sec:numerical-scheme}

In this section, we seek to discretise the SDEs \eqref{eq:sdes} via a numerical scheme which is stochastically monotone; in light of \Cref{rem:suff_cond_stab_lp} and \Cref{prop:kr-discrete-optimality}, this will lead to a discrete-time bicausal transport problem, for which an optimiser is known, and whose value converges to that of our continuous-time problem. 
We make the following observation.

\begin{remark}\label{rem:em}
	Recall the Euler--Maruyama scheme for the SDE \eqref{eq:sde} with coefficients $b\colon \R \to \R$, $\sigma \colon \R \to \R_+$ driven by a Wiener process $W$, as defined in \eqref{eq:approx}:
	
	For $N \in \N$, let $h = 1/N$ and $X^h_0 = x_0$. Then, for each $k = 0, \dotsc, N - 1$ and $t \in (kh, (k + 1)h]$, define
	\begin{equation}\label{eq:em}
		X^h_t = X^h_{kh} + (t - kh) b(X^h_{kh}) + \sigma(X^h_{kh}) (W_t - W_{kh}).
	\end{equation}
	We call the process $(X^h_t)_{t \in [0, 1]}$ the \emph{Euler--Maruyama scheme} and refer to $(X^h_{kh})_{k \in \{0, \dotsc, N - 1\}}$ as the \emph{discrete-time Euler--Maruyama scheme}.	
	
	Now suppose that $b$ is Lipschitz, with Lipschitz constant $C$, and that $\sigma$ is \emph{constant}. Then, for $h < C^{-1}$, the process $(X^h_{kh})_{k \in \{0, \dotsc, N - 1\}}$ is stochastically increasing. Indeed the function $x \mapsto x + h b(x)$ is increasing and so, for $a \in \R$ and $x < x^\prime$,
$$\P\!\left[x + h b(x) + \sigma (W_{(k+1)h} - W_{kh}) \leq a\right] \geq \P \!\left[x^\prime + h b(x^\prime) + \sigma(W_{(k + 1)h} - W_{kh})\le a\right].$$
\end{remark}

When we take a non-constant diffusion coefficient $\sigma$, the above discrete-time Euler--Maruyama scheme may no longer be stochastically monotone. We therefore define the following variant of the Euler--Maruyama scheme, in which the Brownian increments are truncated, such that we recover the desired stochastic monotonicity.

Let $W$ be a standard one-dimensional Wiener process, let $N \in \N$, $h = 1/N$, and define the truncation level $A_h \coloneqq  2\sqrt{- h \log h}$. Let $W^h_0=0$. For each $k = 0, 1, \dotsc$, define the stopping time $\tau^h_k$ to be the first time after $kh$ that the Brownian increment $(W_\cdot - W_{kh})$ leaves the interval $(-A_h, A_h)$, and define \begin{align}
 			W^h_t\coloneqq  W^h_{kh} + (W_\cdot -W_{kh})_{t \wedge \tau^h_k}, \quad \text{for} \quad t \in (kh, (k + 1)h].
	\end{align}

\begin{definition}[monotone Euler--Maruyama scheme]\label{def:monotone-em}
	Consider the SDE \eqref{eq:sde} for some coefficients $b\colon \R \to \R$, $\sigma \colon \R \to \R_+$ and a Wiener process $W$. Fix $N \in \N$, $h = \frac{1}{N}$, and define $W^h$ as above.
	Let $X^h_0 = x_0$. Then, for each $k = 0, \dotsc, N - 1$ and $t \in (kh, (k + 1)h]$, define
\begin{equation}
	X^h_t\coloneqq  X^h_{kh} + (t - kh) b(X^h_{kh}) + \sigma(X^h_{kh}) (W^h_t - W^h_{kh}).
\end{equation}
We call the process $(X^h_t)_{t \in [0, 1]}$ the \emph{monotone Euler--Maruyama scheme} and refer to $(X^h_{kh})_{k \in \{0, \dotsc, N - 1\}}$ as the \emph{discrete-time monotone Euler--Maruyama scheme}.
\end{definition}

We now verify that, for Lipschitz coefficients and sufficiently small $h > 0$, the discrete-time monotone Euler--Maruyama scheme is stochastically monotone.

\begin{lemma}\label{lem:truncated-increasing-kernels}
	Suppose that the coefficients $b$ and $\sigma$ in \eqref{eq:sde} are Lipschitz. Then, for sufficiently small $h > 0$, the discrete-time monotone Euler--Maruyama scheme $(X^h_{kh})_{k \in \{0, \dotsc, N - 1\}}$ for \eqref{eq:sde} is stochastically increasing.
\end{lemma}

\begin{proof} 
	Let $x, x^\prime \in \R$ such that $x < x^\prime$ and let $k\in\{0,\dots,N-2\}$. Define the random variables $Y = x + h b(x) + \sigma(x)(W^h_{(k + 1)h} -W^h_{kh})$ and $Y^\prime = x^\prime + h b(x^\prime) + \sigma(x^\prime)(W^h_{(k+1)h} - W^h_{kh})$. Then, letting $C_0$ and $C_1$ be the Lipschitz constants of $b$ and $\sigma$, respectively, and using the bound on the truncated Brownian increment, we have
	\begin{equation}
		Y^\prime - Y \geq (1 - h C_0 - A_h C_1)(x^\prime - x).
	\end{equation}
	Noting that $\lim_{h \to 0} A_h = 0$, we can choose $h$ sufficiently small that $1 - h C_0 - A_h C_1 > 0$, and conclude that we have the desired ordering, in first order stochastic dominance, of $Y$ and $Y^\prime$.
\end{proof}

Combined with \Cref{prop:kr-discrete-optimality}, \Cref{lem:truncated-increasing-kernels} implies that, for the adapted Wasserstein distance between the laws of two discrete-time monotone Euler--Maruyama schemes, the Knothe--Rosenblatt rearrangement is an optimiser, when all coefficients are Lipschitz.
Next we show that, for two SDEs driven by a common Wiener process, the joint law of the discrete-time monotone Euler--Maruyama schemes coincides with the Knothe--Rosenblatt rearrangement between the laws of the two schemes.
	
	\begin{lemma}\label{lem:kr-joint-law}
		Fix a Wiener process $W$ and consider the SDEs \eqref{eq:sdes} driven by the common Wiener process $W$ --- i.e.\ $\bar{W} = W$ in \eqref{eq:sdes}. 
		For $h > 0$, let $X^h, \bar{X}^h$ be the associated discrete-time monotone Euler--Maruyama schemes, and write $\mu^h, \nu^h$ for their respective laws. Then the joint law $\Law(X^h, \bar{X}^h)$ is equal to the Knothe--Rosenblatt rearrangement $\pi^\kr_{\mu^h, \nu^h}$ between $\mu^h$ and $\nu^h$.
	\end{lemma}
	
	\begin{proof}
		Let $N \in \N$ and $h = 1/N$. Fix $k \in \{1, \dotsc, N - 1\}$ and write $\Delta W^h_{k+1}$ for the truncated increment $W^h_{(k + 1)h} - W^h_{kh}$ of the Wiener process. Then, since $\sigma, \bar{\sigma}$ are non-negative functions, the maps $\Delta W^h_{k+1} \mapsto X^h_{(k + 1)h}$, given by $X^h_{(k+1)h} = X^h_{kh} + h b(X^h_{kh}) + \sigma(X^h_{kh}) \Delta W^h_{k+1}$, and $\Delta W^h_{k+1} \mapsto \bar{X}^h_{(k+1)h}$, given by $\bar X^h_{(k+1)h} = \bar X^h_{kh} + h\bar b(\bar X^h_{kh}) + \bar \sigma(\bar X^h_{kh}) \Delta W^h_{k+1}$, are both increasing. Since we take the same random variables $\Delta W^h_{k+1}$, $k = 1, \dotsc, N-1$, for both $X^h$ and $\bar{X}^h$, the characterisation of the Knothe--Rosenblatt rearrangement given in \Cref{rem:kr-increasing-map} implies that the joint law of $X^h$ and $\bar{X}^h$ coincides with the Knothe--Rosenblatt rearrangement between $\mu^h$ and $\nu^h$.
	\end{proof}

	We finally establish that the monotone scheme converges in the $L^p$-norm to the solution of the SDE; the proof is deferred to \Cref{app:convergence}. 

\begin{proposition}\label{lem:convergence-truncated}
	Let $p\ge 1$.
	Suppose that the coefficients $b$ and $\sigma$ in \eqref{eq:sde} are Lipschitz and let $X$ denote its unique strong solution.
	Consider the associated monotone Euler--Maruyama scheme $X^h$ given in \Cref{def:monotone-em}. 
	Then we have the $L^p$-convergence
	\begin{equation}
		\lim_{h \to 0}\E \!\left[\sup_{0 \leq t \leq 1} |X^h_t - X_t|^p \right] = 0,
	\end{equation}
	and the associated discrete-time monotone Euler--Maruyama scheme converges to $X$ in $L^p$, in the sense that
\[
\lim_{h\to 0}
\E \!\left[\sum_{k = 0}^{N - 1} \int_{kh}^{(k+1)h}\lvert X^h_{kh} - X_t \rvert^p \D t\right]
=0.
\]
\end{proposition}

\begin{remark}
	\citet{MiReTr02} proved $L^p$-convergence of a Milstein scheme with a truncated Gaussian driving noise, for SDEs with sufficiently regular coefficients. A fully implicit Euler--Maruyama scheme with the same truncated noise was also introduced in \cite{MiReTr02}, but no general convergence result was given.	
	Seemingly independently, \citet{LiPa22} proved $L^p$-convergence of an Euler--Maruyama scheme with truncated Brownian increments for McKean--Vlasov equations. Since the first version of the present article appeared online, \citet{JoPa23} also proved the $L^p$-convergence of a similar Euler--Maruyama scheme with truncated Brownian increments, allowing also for time-dependent coefficients. In \cite{LiPa22,JoPa23}, the authors exploit a monotonicity property similar to \Cref{lem:truncated-increasing-kernels} in order to study the (monotone) convex ordering of continuous-time processes.
	
	We note that, as in \cite{MiReTr02}, the truncation level $A_h$ is chosen in such a way that the moments of the error introduced by the truncation decay sufficiently fast as $h \to 0$.
\end{remark}

\begin{remark}\label{rem:stoch-mono-fdd-new}
    Supposing that the coefficients of the SDE \eqref{eq:sde} are Lipschitz, the unique strong solution $X$ is a Feller process and so, as noted in \Cref{rem:stoch-mono-fdd}, the finite-dimensional distributions of $X$ are stochastically increasing. We sketch an alternative proof of this fact based on approximation by the monotone Euler--Maruyama scheme. For $h$ sufficiently small, the discrete-time monotone Euler--Maruyama scheme $X^h$ is stochastically increasing by \Cref{lem:truncated-increasing-kernels}. For $t_1, t_2 \in [0,1] \cap \Q$ with $t_1 < t_2$, choose $h$ such that $t_1 = k_1 h$, $t_2 = k_2 h$, for some $k_1, k_2 \in \N_0$. Then the transition kernels for $X^h$ from $t_1$ to $t_2$ are also stochastically increasing. The uniform-in-time $L^p$-convergence of $X^h$ to $X$ shown in \Cref{lem:convergence-truncated} implies convergence of finite-dimensional distributions and therefore also of transition kernels. Refining the discretisation grid of $X^h$ appropriately, we conclude that the transition kernels for $X$ from $t_1$ to $t_2$ are stochastically increasing, as required.
\end{remark}

We are now ready to prove the following result, which implies the conclusion of \Cref{thm:synchronous_optimal} when the coefficients in \eqref{eq:sdes} are Lipschitz. 

\begin{proposition}\label{prop:bnt}
	Let $b, \bar{b}\colon \R \to \R$ and $\sigma, \bar{\sigma}\colon \R \to \R_+$ be Lipschitz continuous.
	For $N \in \N$, set $h = 1/N$ and let $\mu^h, \nu^h$ be the laws of the discrete-time monotone Euler--Maruyama schemes for \eqref{eq:sdes}.
	Then, for $c \colon \mathbb R\times \mathbb R \to\mathbb R$ continuous and satisfying \eqref{eq:p_growth_c} and \eqref{eq:L_superadditivity} for some $p \ge 1$ and $K>0$,
\[
\lim_{h\to 0}\;
\inf_{\pi\in\cplbc\!\left(\mu^h,\nu^h\right)}h\int\sum_{k=0}^{N - 1} c\!\left(x_k, \bar x_k\right)\di\pi\;
=
\inf_{\pi\in\cplbc(\mu^{b, \sigma},\mu^{\bar b, \bar \sigma})}\iint_0^1c\!\left(\omega_t,\bar{\omega}_t\right)\dt\di\pi;
\]
	moreover, the infimum on the right hand side is attained by the synchronous coupling.
\end{proposition}

\begin{remark}\label{rem:lipschitz-well-defined}
	Note that, under the Lipschitz conditions on the coefficients, there exist $p$-integrable unique strong solutions of the SDEs \eqref{eq:sdes}, for $p \geq 1$; see \Cref{rem:ito-strong-solution}. Also, according to \Cref{lem:truncated-scheme-l2-bound}, the associated monotone Euler--Maruyama schemes are bounded in $L^p$, $p \geq 1$. The bicausal optimal transport problems in the statement of \Cref{prop:bnt} are therefore well-defined.
\end{remark}

\begin{proof}[Proof of \Cref{prop:bnt}]
Let $X^h, \bar X^h \colon \Omega \to \hat \Omega$ be measurable maps defining constant interpolations of the monotone Euler--Maruyama schemes for \eqref{eq:sdes}. That is, for any $k \in \{0, \dotsc, N - 1\}$ and $t \in [kh, (k + 1)h)$, and for any Wiener process $W$, $(X^h \circ W)_t$ and $(\bar X^h \circ W)_t$ coincide with the respective monotone Euler--Mauryama schemes driven by $W$ at time $kh$. Then, for any correlated Wiener process $(W, \bar W)$, defining $\pi \coloneqq \Law(((X^h \circ W)_{kh})_{k \in \{0, \dotsc, N - 1\}}, ((\bar X^h \circ \bar W)_{kh})_{k \in \{0, \dotsc, N - 1\}})$, we have
\begin{equation}\label{eq:bicausal-correlated-proof}
	\int h \sum_{k = 0}^{N - 1} c(x_k, y_k) \pi (\D x, \D y) = \E\!\left[\int_0^1 c((X^h \circ W)_t, (\bar X^h \circ \bar W)_t) \D t\right],
\end{equation}
and $\pi \in \cplbc(\mu^h, \nu^h)$. By \Cref{lem:truncated-increasing-kernels} and \Cref{prop:kr-discrete-optimality}, for $h > 0$ sufficiently small, the Knothe--Rosenblatt rearrangement $\pi^\kr_{\mu^h, \nu^h}$ attains the infimum
\begin{equation}\label{eq:inf-bc-proof}
	\inf_{\pi\in\cplbc(\mu^h,\nu^h)}\int h \sum_{k = 0}^{N - 1} c(x_k,y_k) \pi(\D x, \D y).
\end{equation}
Moreover, by \Cref{lem:kr-joint-law}, $\pi^\kr_{\mu^h,\nu^h} = \Law(((X^h \circ W)_{kh})_{k \in \{0, \dotsc, N - 1\}}, ((\bar X^h \circ W)_{kh})_{k \in \{0, \dotsc, N - 1\}})$. Hence the perfectly correlated Wiener process $(W, W)$ attains the infimum taken over all correlated Wiener processes on the right-hand side of \eqref{eq:bicausal-correlated-proof}, and this infimum coincides with \eqref{eq:inf-bc-proof}.

By \Cref{lem:convergence-truncated}, for any Wiener process $W$, $X^h \circ W$ converges in $L^p$ to $X^{b, \sigma}$ driven by $W$, and analogously for $\bar X^h$. Hence we can conclude by \Cref{rem:suff_cond_stab_lp}.
\end{proof}

The above result shows that the continuous-time adapted Wasserstein distance between laws of solutions of \eqref{eq:sdes} with Lipschitz coefficients is the limit of discrete-time adapted Wasserstein distances, each of which are attained by the Knothe--Rosenblatt rearrangement; meanwhile, the synchronous coupling attains the continuous-time adapted Wasserstein distance. In fact, we also have convergence of the optimisers. In view of \Cref{thm:convergence-general} and the $L^p$-convergence that we prove in \Cref{lem:convergence-truncated}, our particular choice of continuous-time scheme leads to the following convergence of the Knothe--Rosenblatt rearrangement to the synchronous coupling in the adapted Wasserstein distance. For this reason, we argue that the synchronous coupling can be viewed as a continuous-time analogue of the Knothe--Rosenblatt rearrangement.

\begin{proposition}\label{prop:aw-conv-kr-sync}
   Let $b, \bar{b}\colon \R \to \R$ and $\sigma, \bar{\sigma}\colon \R \to \R_+$ be Lipschitz continuous.
   Consider the SDEs \eqref{eq:sdes} driven by a common Wiener process $W$, and write $X, \bar X$ for the solutions of \eqref{eq:sdes} and $\mu, \nu$ for their laws. For any $N \in \N$ and $h = 1/N$, let $X^h, \bar X^h$ be the monotone Euler--Maruyama schemes for \eqref{eq:sdes} driven by the common Wiener process $W$. Further, let $\mu^h = \Law((X^h_{kh})_{k \in \{0, \dotsc, N - 1\}})$, $\nu^h = \Law((\bar X^h_{kh})_{k \in \{0, \dotsc, N - 1\}})$ denote the laws of the discrete-time monotone Euler--Maruyama schemes.
   
   Then, for any $N \in \N$ and $h = 1/N$, $\Law(X^h, \bar X^h)$ is an interpolation of the Knothe--Rosenblatt rearrangement, in the sense that
    \begin{equation}\label{eq:kr-interpolation}
        \Law\big((X^h_{kh}, \bar X^h_{kh})_{k \in \{0. \dotsc, N - 1\}}\big) = \pi^\kr_{\mu^h, \nu^h}.
    \end{equation}
    Moreover, for any $p \geq 1$,
    \begin{equation}\label{eq:aw-conv-kr-sync}
        \lim_{h \to 0}\AW_p\big(\Law(X^h, \bar X^h), \pi^\sync_{\mu, \nu}\big) = 0.
    \end{equation}
\end{proposition}

\begin{proof}
    By definition of the synchronous coupling, $\Law(X, \bar X) = \pi^\sync_{\mu, \nu}$. For fixed $N \in \N$ and $h = 1/N$, \Cref{lem:kr-joint-law} implies that $\Law(X^h, \bar X^h)$ interpolates the Knothe--Rosenblatt rearrangement in the sense of \eqref{eq:kr-interpolation}.

    For any $p \geq 1$, we have the $L^p$ convergence $X^h \to X$ and $\bar X^h \to \bar X$ by \Cref{lem:convergence-truncated}. Since the same correlated Wiener process $(W, W)$ is chosen to drive both the SDEs \eqref{eq:sdes} and the monotone Euler--Maruyama schemes, the desired $\AW_p$ convergence \eqref{eq:aw-conv-kr-sync} then follows from \Cref{thm:convergence-general}.
\end{proof}

\subsection{$\AW_p$ between laws of SDEs with continuous coefficients}\label{sec:continuous-general}

In this section, we complete the proof of \Cref{thm:synchronous_optimal} under \Cref{ass:main}; that is, we relax the above Lipschitz assumption and show optimality of the synchronous coupling for the adapted Wasserstein distance between laws of SDEs for which pathwise uniqueness holds and whose coefficients are continuous and have linear growth. 

We start by making some remarks on \Cref{ass:main} and providing examples of coefficients that satisfy this assumption.

\begin{remark}\label{rem:strong-existence}
	Under \Cref{ass:main}, strong existence is guaranteed for the SDEs \eqref{eq:sdes}. Indeed, by a result of Skorokhod \cite{Sk65}, there exist weak solutions of the SDEs under the given continuity and linear growth assumptions on the coefficients. Then, by the Yamada-Watanabe criterion \cite[Ch.\ 5, Corollary 3.23]{KaSh91}, the combination of pathwise uniqueness and weak existence implies the existence of strong solutions. We refer to \cite{CoRo22} for an example of a Markovian SDE for which strong existence does not hold.
\end{remark}

\begin{remark}\label{rem:main_ass}
	\Cref{ass:main} is satisfied, for example, in the following cases:
	\begin{enumerate}[label=(\roman*)]
		\item $b, \bar{b}, \sigma, \bar{\sigma}$ are Lipschitz \cite[Proposition 1.9 (It\^o)]{ChEn05} --- see \Cref{prop:bnt};
		\item $b, \bar{b}, \sigma, \bar{\sigma}$ are continuous and bounded, $\sigma, \bar{\sigma}$ are $1/2$-H\"older continuous and bounded below by a positive constant \cite[Proposition 1.10 (Zvonkin)]{ChEn05} --- see \Cref{prop:zvonkin};
		\item $b, \bar{b}, \sigma, \bar{\sigma}$ are continuous with linear growth, $\sigma, \bar{\sigma}$ are strictly positive and $1/2$-H\"older continuous, and $b/\sigma^2, \bar{b}/\bar{\sigma}^2$ are locally Lebesgue-integrable \cite[Proposition 1.11 (Engelbert--Schmidt)]{ChEn05};
		\item $b, \bar{b}$ are Lipschitz, $\sigma, \bar{\sigma}$ have linear growth and are uniformly continuous with a strictly increasing modulus of continuity $h \colon \R_+ \to \R$ satisfying $\int_0^{0+}h^{-2}(x) \D x = +\infty$ \cite[Proposition 1.12 (Yamada--Watanabe)]{ChEn05}.
	\end{enumerate}
\end{remark}

We are now ready to provide the following theorem, which, in particular, implies our main result, \Cref{thm:synchronous_optimal}.

	\begin{theorem}\label{thm:optimality_general_c}
		Suppose that $(b, \sigma)$ and $(\bar b, \bar \sigma)$ satisfy \Cref{ass:main}.
		Let $c \colon \mathbb R\times \mathbb R \to\mathbb R$ be continuous and satisfy \eqref{eq:p_growth_c} and \eqref{eq:L_superadditivity}, for some $p \ge 1$ and $K>0$. 
		Then the synchronous coupling attains the following infimum:
		\[
\inf_{\pi\in\cplbc(\mu^{b,\sigma},\mu^{\bar b,\bar\sigma})}\iint_0^1c\!\left(\omega_t,\bar{\omega}_t\right)\dt\di\pi.
\]
	If the inequality in \eqref{eq:L_superadditivity} is reversed, then this infimum is attained by the anti-synchronous coupling. 
	\end{theorem}	
	
\begin{proof}
	Under Lipschitz conditions on the coefficients of the SDEs \eqref{eq:sdes}, we have already proved the conclusion of the theorem in \Cref{prop:bnt}. Now suppose that the more general condition of \Cref{ass:main} is satisfied, namely that the coefficients are continuous with linear growth, and that pathwise uniqueness holds. Then, according to \Cref{rem:corr_uniqueness}, the SDEs \eqref{eq:sdes} satisfy \Cref{ass:stability}, where the coefficients $b, \bar{b}, \sigma, \bar{\sigma}$ now are Markovian and time-homogeneous and we identify, for example, $b(t, \omega) = b(\omega_t)$, $t \in [0, 1]$, $\omega \in \Omega$. We can approximate $(b, \bar{b}, \sigma, \bar{\sigma})$ locally uniformly by a sequence of Lipschitz functions $(b^n, \bar{b}^n, \sigma^n, \bar{\sigma}^n)_{n \in \N}$ on $\R$, which all satisfy the same linear growth bound. Note that locally uniform convergence of the Markovian coefficients implies convergence in the sense of \eqref{eq:unifconvassump}. Defining $\mu^n$, $\nu^n$ as the laws of $X^{b^n, \sigma^n}$, $X^{\bar b^n, \bar \sigma^n}$, we have by \Cref{prop:bnt} that the synchronous coupling $\pi^\sync_{\mu^n, \nu^n}$ attains the infimum for the corresponding bicausal transport problem. The first part thus follows by \Cref{cor:synchronous-stability}.
	Finally, if $c$ satisfies \eqref{eq:L_superadditivity} with the inequality reversed, then \Cref{prop:kr-discrete-optimality} and \Cref{lem:kr-joint-law}, and thus \Cref{prop:bnt}, hold with obvious modifications and the last part then follows by use of the same arguments as used above. 
\end{proof}

\begin{remark}\label{rem:sync-uniqueness}
	For the case of $\AW_2$ with sufficiently regular coefficients, uniqueness of the optimal coupling follows from \cite{BiTa19}, in view of \Cref{rem:BN-T:distance}. Namely, when $b, \bar b, \sigma, \bar \sigma$ are continuously differentiable with $\alpha$-H\"older first derivative for some $\alpha \in (0, 1)$, and $\sigma, \bar \sigma$ are bounded away from zero, \cite[Section 2.1]{BiTa19} gives a PDE formulation of $\AW_2$, from which we can deduce uniqueness of the optimiser. Recovering and extending this uniqueness result via probabilistic arguments is left open for future research.
\end{remark}

\begin{remark}
	We note that we can further extend \Cref{thm:synchronous_optimal} by combining different sets of assumptions. If the coefficients $(b, \sigma)$ satisfy \Cref{ass:main} and the coefficients $(\bar b, \bar \sigma)$ satisfy \Cref{ass:zvonkin} (or vice-versa), then the conclusion of \Cref{thm:synchronous_optimal} still holds, by the following reasoning.
	
	Suppose that $(\bar b, \bar \sigma)$ satisfy \Cref{ass:zvonkin}. If $(b, \sigma)$ are Lipschitz, then by examining the proofs of \Cref{prop:bnt} and \Cref{prop:zvonkin}, it is straightforward to see that the result still holds.
	
	Now suppose that $(b, \sigma)$ satisfy \Cref{ass:main} but are not Lipschitz. We then need to adapt the stability result of \Cref{prop:stab_abstract} because the coefficients $(\bar b, \bar \sigma)$ may not be continuous. As before, we can approximate $(b, \sigma)$ by Lipschitz functions in the sense of \eqref{eq:unifconvassump}. On the other hand, we fix the constant sequence $(\bar b^n, \bar \sigma^n) = (\bar b, \bar \sigma)$, for all $n \in \N$. Then, after applying Skorokhod's representation theorem, we apply Lusin's theorem for a second time, in order to find a continuous function that coincides with $\bar b$ on a set of arbitrarily large measure, with respect to the law of the fixed process $X^{\bar b, \bar \sigma}$. We can then follow the remainder of the proof as above.
\end{remark}

\begin{remark}[time-dependent coefficients]
	One can also extend \Cref{thm:synchronous_optimal} to the time-inhomogeneous case. In particular, assuming that the coefficients are Lipschitz in space uniformly in time, Lipschitz in time uniformly in space, and have linear growth in space, the proofs in \Cref{sec:formulation} and \Cref{sec:proofs} remain valid with only minor modifications. In this case, the monotone Euler--Maruyama scheme should be modified such that the expression for $X^h_t$ in \Cref{def:monotone-em} is replaced by $X^h_t = X^h_{kh} + (t - kh)b(kh, X^h_{kh}) + \sigma(kh, X^h_{kh})(W^h_t - W^h_{kh})$. Under the given conditions on the coefficients, the corresponding standard Euler--Maruyama scheme converges in $L^2$ by \cite[Theorem 10.2.2]{KlPl92} and one can deduce $L^p$ convergence of the monotone scheme, for any $p \geq 1$, as in \Cref{sec:numerical-scheme}. Further, the above stability argument still applies and also allows us to pass to coefficients that are only continuous in time. Therefore \Cref{thm:synchronous_optimal} holds under \Cref{ass:main} for time-dependent coefficients that are also continuous in time.
\end{remark}

\subsection{Extension to discontinuous drifts}\label{sec:discontinuous}

In this section, we establish the conclusion of \Cref{thm:synchronous_optimal} under a different set of assumptions, which allows for discontinuities in the drift. For further results in this direction, see the follow-up work \cite{RoSz24}.
For the coefficients under consideration, we are able to employ a similar approach as used for the case of Lipschitz coefficients above.
Specifically, we first apply a Zvonkin-type transformation to remove the drift and then use the monotone Euler--Maruyama scheme for the resulting martingales which feature Lipschitz diffusion coefficients. 
To this end, we work under the following assumption.

\begin{assumption}\label{ass:zvonkin}
	Suppose that the coefficients $b, \bar{b}, \sigma, \bar{\sigma}$ of the SDEs \eqref{eq:sdes} satisfy the following conditions:
	\begin{enumerate}[label=(\roman*)]
		\item $b, \bar{b}$ are bounded and measurable;
		\item $\sigma, \bar{\sigma}$ are bounded, uniformly positive and Lipschitz continuous; \emph{and}
		\item $b/\sigma^2, \bar{b}/\bar{\sigma}^2$ are Lebesgue-integrable.
	\end{enumerate}
\end{assumption}

\begin{remark}
	By Zvonkin's theorem \cite{Zv74}, there exist unique strong solutions $(X_t)_{t \geq 0}, (\bar X_t)_{t \geq 0}$ of the SDEs \eqref{eq:sdes} under \Cref{ass:zvonkin} (i)--(ii). In fact, for well-posedness of the SDEs, the Lipschitz continuity can be weakened to $1/2$-H\"older continuity, but we will make use of the Lipschitz condition later on in order to apply the monotone Euler--Maruyama scheme.	
\end{remark}

\begin{proposition}\label{prop:zvonkin}
		Suppose that $(b, \sigma)$ and $(\bar b, \bar \sigma)$ satisfy \Cref{ass:zvonkin}.
		Let $c \colon \mathbb R\times \mathbb R \to\mathbb R$ be continuous and satisfy \eqref{eq:p_growth_c} and \eqref{eq:L_superadditivity}, for some $p \ge 1$ and $K>0$. 
		Then the synchronous coupling attains the following infimum:
		\[
\inf_{\pi\in\cplbc(\mu^{b,\sigma},\mu^{\bar b,\bar\sigma})}\iint_0^1c\!\left(\omega_t,\bar{\omega}_t\right)\dt\di\pi.
\]
In particular, for any $p \ge 1$, the synchronous coupling attains the infimum in $\AW_p(\mu^{b,\sigma},\mu^{\bar b,\bar\sigma})$.
	\end{proposition}

To prove this result, we use exactly the drift-removing transformation
that Zvonkin introduced to prove existence and uniqueness of
strong solutions in~\cite{Zv74}. Define the increasing map $T\colon \R \to \R_+$ by
\begin{equation}
	T(x)\coloneqq  \int_{x_0}^x\exp\!\left\{- 2 \int_{x_0}^z \frac{b(y)}{\sigma^2(y)} \D y \right\}\D z, \quad x \in \R,
\end{equation}
and let $Y_t\coloneqq  T(X_t)$, for $t \in [0, 1]$, where $X$ is the unique strong solution of \eqref{eq:sde} with coefficients $b, \sigma$. Then, by It\^o's formula, $Y$ solves the SDE
\begin{equation}\label{eq:sde-transformed}
	\D Y_t = (\sigma T^\prime) \circ T^{-1}(Y_t) \D W_t; \quad Y_0 = T(x_0).
\end{equation}

\begin{lemma}\label{lem:zvonkin-lipschitz}
	Suppose that $b \colon \R \to \R$ is bounded and measurable, and that $\sigma \colon \R \to \R_+$ is bounded, uniformly positive, and Lipschitz. Then the map $(\sigma T^\prime) \circ T^{-1} \colon \R \to \R_+$ is Lipschitz.
\end{lemma}

\begin{proof}
	For $x_1, x_2 \in \R$, we are required to show that there is some constant $K > 0$ such that
	\begin{equation}
		\lvert \sigma(x_2) T^\prime(x_2) - \sigma(x_1) T^\prime(x_1) \rvert \leq K \lvert T(x_2) - T(x_1)\rvert.
	\end{equation}
	Since $\sigma$ is Lipschitz, there exists a Lebesgue-almost everywhere derivative $\sigma^\prime$ that is Lebesgue-almost surely bounded by the Lipschitz constant $K^\sigma$ of $\sigma$. Hence $\sigma T^\prime$ is also Lebesgue-almost surely differentiable, and its derivative satisfies
	\begin{equation}
		\begin{split}
			(\sigma T^\prime)^\prime(x)
			& = \sigma^\prime(x) \exp\!\left\{ - 2 \int_{x_0}^x \frac{b(y)}{\sigma^2(y)} \D y \right\} - 2 \frac{b(x)}{\sigma(x)}\exp\!\left\{- 2 \int_{x_0}^x \frac{b(y)}{\sigma^2(y)} \D y\right\},
		\end{split}
	\end{equation}
	for Lebesgue-almost every $x \in \R$. Then, integrating, we have
	\begin{equation}
		\begin{split}
			\!\left\lvert \sigma(x_2)T^\prime(x_2) - \sigma(x_1)T^\prime(x_1)\right \rvert
			& \leq \!\left(K^\sigma + 2 \frac{\lVert b\rVert_\infty}{\inf_{y \in \R}\sigma(y)}\right)\lvert T(x_2) - T(x_1)\rvert, 
		\end{split}
	\end{equation}
	using the Lebesgue-almost sure bound on $\sigma^\prime$ and the assumption that $\sigma$ is bounded away from zero.
\end{proof}

In light of \Cref{lem:zvonkin-lipschitz}, we can apply the monotone Euler--Maruyama scheme defined in \Cref{def:monotone-em} to the transformed SDE \eqref{eq:sde-transformed}. Fix $N \in \N$ and $h = 1/N$. The full scheme for the SDE \eqref{eq:sde} in this case is then as follows. Let $X^h_0 = x_0$ and, for $k = 0, \dotsc, N - 1$ and $t \in (kh, (k+1)h]$, define
\begin{equation}\label{eq:transformed-em}
	X^h_t\coloneqq  T^{-1}\!\left[T(X^h_{kh}) + \sigma(X^h_{kh})T^{\prime}(X^h_{kh})(W^h_t - W^h_{kh})\right].
\end{equation}

The map $T$ is increasing and invertible with increasing inverse. Therefore, for $h > 0$ sufficiently small and $k = 0, \dotsc, N - 1$, we can use the same arguments as in the proof of \Cref{lem:truncated-increasing-kernels}, along with the fact that $\sigma T^\prime \circ T^{-1}$ is Lipschitz, to see that $X^h_{kh} \mapsto X^h_{(k+1)h}$ is a concatenation of three increasing maps. Hence the process $(X^h_{kh})_{k = 0, \dotsc, N}$ is stochastically increasing.

\begin{remark}\label{rem:discontinuous}
	\Cref{ass:zvonkin}.(iii) guarantees that $T^{-1}$ is Lipschitz, with some Lipschitz constant $C > 0$. Indeed, if $x < x^\prime$, then
	\begin{equation}
		|T(x^\prime) - T(x)| = \int_x^{x^\prime} \exp\!\left\{ \int_{x_0}^z -2\frac{b(y)}{\sigma^2(y)} \D y \right\} \D z \geq (x^\prime - x) \exp\!\left\{ -2 \sup_{z \in \R}\int_{x_0}^z \frac{b(y)}{\sigma^2(y)} \D y\right\}.
	\end{equation}
	
	Applying the moment bound \eqref{eq:lp-ito} to the solution $Y$ of \eqref{eq:sde-transformed} thus gives us that the solution $X$ of \eqref{eq:sde} is also $p$-integrable, for $p \ge 1$. Similarly, applying \Cref{lem:truncated-scheme-l2-bound} to the monotone Euler--Maruyama scheme $Y^h$ for \eqref{eq:sde-transformed} yields an $L^p$-bound for $X^h$ defined by \eqref{eq:transformed-em}, for $h > 0$ and $p \geq 1$. 
	The bicausal optimal transport problems appearing in \Cref{prop:zvonkin} are therefore well-defined.
	
	Furthermore, for any $p \ge 1$, $h > 0$ and $s \in [0, 1]$, we can write
	\begin{equation}
		\lvert X^h_s - X_s \rvert^p = \lvert T^{-1}(Y^h_s) - T^{-1}(Y_s)\rvert^p \leq C \lvert Y^h_s - Y_s\rvert^p.
	\end{equation}
	Hence, from the $L^p$-convergence of $Y^h$ to $Y$ given by \Cref{lem:convergence-truncated}, we can deduce $L^p$-convergence of $X^h$ to $X$.
\end{remark}

\begin{proof}[Proof of \Cref{prop:zvonkin}]
Thanks to \Cref{lem:zvonkin-lipschitz} and \Cref{rem:discontinuous}, the result follows by use of the same arguments as used to prove \Cref{prop:bnt}.
\end{proof}

\begin{remark}
Since our stability result in \Cref{cor:synchronous-stability} requires the coefficients to be continuous, with the methods employed in this paper, \Cref{prop:zvonkin} is the most general result that we are able to obtain for SDEs with discontinuous coefficients.
\end{remark}

\section{On the topology induced by the adapted Wasserstein distance}\label{sec:topology}
In this section, we apply the stability result of \Cref{prop:stab_abstract} to prove \Cref{thm:topologies}. This theorem states that, restricted to a particular subset of probability measures, the topology induced by the adapted Wasserstein distance coincides with the topologies induced by the synchronous distance, the (symmetric) causal Wasserstein distance and the classical Wasserstein distance, as defined in the introduction, as well as with the topologies of weak convergence and convergence in finite-dimensional distributions.

We recall the following notation. When strong existence and pathwise uniqueness hold for the SDE \eqref{eq:sde} with coefficients $(b, \sigma)$, we write $X^{b, \sigma}$ for the unique strong solution, and $\mu^{b, \sigma}$ for its law. We then define $\Pc^\ast$ to be the set of all such laws. For $\Lambda > 0$, we define $\A^\Lambda$ to be the set of $\Lambda$-Lipschitz functions whose absolute value at zero is also bounded by $\Lambda$, and then define $\Pc^\Lambda\coloneqq  \{\mu^{b, \sigma} \colon \; (b, \sigma) \in \A^\Lambda \times \A^\Lambda\} \subset \Pc^\ast$. A further result of \Cref{thm:topologies} is that the set $\Pc^\Lambda$ is compact with respect to the topology induced by the adapted Wasserstein distance.

We introduce a further subset
	$		\bar \Pc \coloneqq \big\{
		\mu^{b,\sigma}:
		\textrm{$(b,\sigma)$ satisfies \Cref{ass:main}}
		\big\} \subset \Pc^\ast$.
Also recall the set $\Pc_p$ of measures on $\Omega$ with finite $p$\textsuperscript{th} moment, for $p \geq 1$. We note that $\sup_{t \in [0, 1]}\omega_t\in L^p(\mu)$, for any $\mu\in\bar \Pc$ and $p\ge 1$, by the estimates given in \Cref{rem:ito-strong-solution}.
	For $p\ge 1$, the synchronous distance $\SW_p$ is therefore well-defined on $\bar \Pc$, and $(\bar \Pc,\SW_p)$ is a metric space.

\begin{proof}[Proof of \Cref{thm:topologies}]
		Note first that, for $p \geq 1$, we have the inclusion $\Pc^\Lambda\subset\bar \Pc\subset\Pc_p$. The subspace topologies obtained by restricting the topologies listed in the statement of the theorem to $\Pc^\Lambda$ are thus well-defined.	
		
		Let us define another (a priori stronger) topology $\tau_p$ on $\Pc^\Lambda$ to be the topology induced by the distance $\SW^\infty_p$, defined by
		\begin{equation}
			\SW^\infty_p(\mu, \nu)\coloneqq  \E^{\pi^\sync_{\mu, \nu}}\!\left[\sup_{0\le t\le 1}\!\left|\omega_t - \bar \omega_t\right|\!^p\right]^{1/p}, \quad \mu, \nu \in \bar \Pc, \quad p \in [1, \infty).
		\end{equation} We first show that, for $p \in [1, \infty)$, $\Pc^\Lambda$ is (sequentially) compact with respect to $\tau_p$ and that $\tau_p$ is independent of $p$.
		To this end, note that, viewed as a subset of continuous functions, $\A^\Lambda$ is (sequentially) compact with respect to the topology of local uniform convergence. Indeed, consider a sequence $(\varphi_n)_{n\in\mathbb{N}}\subset\A^\Lambda$. For every $K>0$, by the Arzel{\`a}-Ascoli theorem, there exists a subsequence that converges uniformly on $[-K,K]$; it follows by a diagonalisation argument that the sequence $(\varphi_n)_{n\in\mathbb{N}}$ converges locally uniformly and its limit belongs to $\A^\Lambda$. 
		Consider a sequence $(\mu_n)_{n\in\mathbb{N}}\subset\Pc^\Lambda$ and let $(b_n,\sigma_n)_{n \in \N}\subset\A^\Lambda\times\A^\Lambda$ be such that $\mu_n=\mu^{b_n,\sigma_n}$. 
		By the above compactness, an equally denoted subsequence $(b_n, \sigma_n)_{n \in \N}$ converges locally uniformly to some $(b, \sigma) \subset\A^\Lambda\times\A^\Lambda$; we write $\mu=\mu^{b,\sigma}$. 
		Now fix a probability space $(\Omega, \F, \P)$ supporting a Wiener process $W$. Let $X^{b,\sigma}$ and $X^{b_n,\sigma_n}$ be the unique strong solutions of the SDE \eqref{eq:sde} driven by the same Wiener process $W$, on the same probability space, with coefficients $(b, \sigma)$ and $(b_n, \sigma_n)$, respectively, for $n \in \N$.
		Applying \Cref{prop:stab_abstract}, we obtain that $(\Law(X^{b_n,\sigma_n},X^{b,\sigma}))_{n\in\mathbb{N}}$ converges in the $p$-Wasserstein distance on $\Pc(\Omega \times \Omega)$ to $\Law(X^{b,\sigma},X^{b,\sigma})$, for any $p\ge 1$. Noting that, for any $p \in [1, \infty)$, the function $(\omega, \bar \omega) \mapsto \sup_{0 \leq t \leq 1}|\omega_t - \bar \omega_t|^p$ is continuous and has at most rate $p$ polynomial growth, we have by \Cref{rem:wass-p} that,
		\begin{equation}\label{eq:proof_topologies}
		\SW_p^\infty(\mu_n,\mu)
		=
		\E\!\left[\sup_{0\le t\le 1}\!\left|X^{b_n,\sigma_n}_t-X^{b,\sigma}_t\right|\!^p\right]^{1/p}
		\xrightarrow[n\to\infty]{}0,
		\end{equation}
		for any $p \in [1, \infty)$.
		Therefore $\Pc^\Lambda$ is (sequentially) compact w.r.t.\ $\tau_p$ for every $p \in [1, \infty)$ and, by the same argument, all topologies $\tau_p$ on $\Pc^\Lambda$ coincide. Let us call $\tau$ this common topology.
		
		We now use the following well-known fact for topological spaces $(A, \tau^A)$, $(B, \tau^B)$. If $I\colon (A,\tau^A)\to (B,\tau^B)$ is continuous and invertible, $A$ is $\tau^A$-compact, and $\tau^B$ is Hausdorff, then $I^{-1}$ is continuous. Applied to $I$ being the identity map, $A=B$, $\tau^B$ being Polish, and $\tau^B$ weaker than $\tau^A$, this argument shows that if $A$ is $\tau^A$-compact then $\tau^A=\tau^B$.
			
		As $\Pc^\Lambda$ is $\tau$-compact, and by the previous paragraph, it now only remains to argue that convergence in each of the topologies listed in the theorem is implied by convergence in $\SW_p^\infty$, for some $p\in[1,\infty)$. 
		It is clear that $\SW^\infty_p(\mu, \nu) \ge \SW_p(\mu, \nu)$, for any $p \in [1, \infty)$, $\mu, \nu \in \bar \Pc$. Now note that, for $\mu, \mu_n \in \bar \Pc$, $n \in \N$, we have $\pi^{\sync}_{\mu_n,\mu}\in\cplbc(\mu_n,\mu)$, and therefore, for $p\in [ 1,\infty)$,
		\begin{equation}
			\lim_{n \to \infty}\SW^\infty_p(\mu_n, \mu) = 0
			\implies
			\lim_{n\to\infty}\SW_p(\mu_n,\mu)=0
			\implies  
			\lim_{n\to\infty}\AW_p(\mu_n,\mu)=0.
		\end{equation}
		Further, since $\cplbc(\mu,\nu)\subseteq\cplc(\mu,\nu)\subseteq\cpl(\mu,\nu)$, and since $\AW_p$ and $\W_p$ are both symmetric, we immediately get that for $p\ge 1$,
		\begin{equation}
			\AW_p(\mu,\nu)\ge 
			\SCW_p(\mu,\nu)\ge
			\W_p(\mu,\nu),
			\quad\mu,\nu\in\Pc_p,
		\end{equation}
		which yields the corresponding ordering for the topologies induced by these metrics. By the same token, convergence in $\CW_p$ is also implied by convergence in $\SW^\infty_p$.
		The convergence of $\mu_n$ to $\mu$ in $\W_1$ implies weak convergence of $\mu_n$ to $\mu$ with respect to the uniform topology on $C([0,1],\R)$, which in turn implies that convergence holds also for the weak topology associated with the $L^p$-topology on $C([0,1],\R)$, for any $p\in [1,\infty]$. Finally, convergence in finite-dimensional distributions is implied by convergence in the weak topologies above.
\end{proof}

\begin{remark}
	The result of \Cref{thm:topologies} applies also to sets of the form $\{
	\mu^{b,\sigma}\colon (b,\sigma)\in\A^\Lambda\times\A^{\tilde\Lambda}
	\}$, $\Lambda,\tilde\Lambda>0$. 
	An inspection of the proof shows that the result also applies to the set $\{\mu^{b,\sigma}\colon \; b,\sigma\in\A^{\kappa,\Lambda}, \; \sigma>0\}$, where $\Lambda>0$, $\kappa\in[1/2,1]$ and 
$$\A^{\kappa,\Lambda}
=
\!\left\{\varphi\in C(\R,\R)\colon 
|\varphi(x)-\varphi(y)|\le \Lambda |x-y|^\kappa
\textrm{ and }
|\varphi(x)|\le \Lambda(1 + | x|),
\;x,y\in\R
\right\},$$
applying the same Arzel{\`a}-Ascoli argument. In this case, existence and uniqueness of strong solutions is guaranteed by a result of Engelbert and Schmidt \cite[Proposition~1.11]{ChEn05} (c.f.\ \Cref{rem:main_ass}), as the continuity of the coefficients and the strict positivity of $\sigma$ implies that the local integrability condition required for this result is satisfied.
\end{remark}

\begin{remark}\label{rem:multidim-topo}
	In contrast to the optimality results, \Cref{thm:topologies} admits a multidimensional extension. We first observe that the stability result \Cref{prop:stab_abstract} holds in arbitrary dimensions, following the same proof as in dimension one. For two $\R^d$-valued SDEs, we define the synchronous coupling as the joint law of their solutions when they are driven by a common $d$-dimensional Brownian motion, and we define the subset  $\Pc^\Lambda$ of probability measures on $C([0, T], \R^d)$ analogously to the one-dimensional setting. Then one can adapt the proof of \Cref{thm:topologies} to prove the same topological equivalence in general dimensions. We note that this multidimensional result extends \cite[Propositions 1.8 and 1.9]{BiTa19}.
\end{remark}

\section{Examples}\label{sec:discussion}
	In this final section, we collect some examples. We first provide one that motivates the use of the adapted Wasserstein distance when considering distances between processes. This is a continuous-time analogue of the example given in \cite{BaBaBeEd19a}.
	
	\begin{example}[motivating example]\label{ex:motivation}
		\begin{figure}[h]
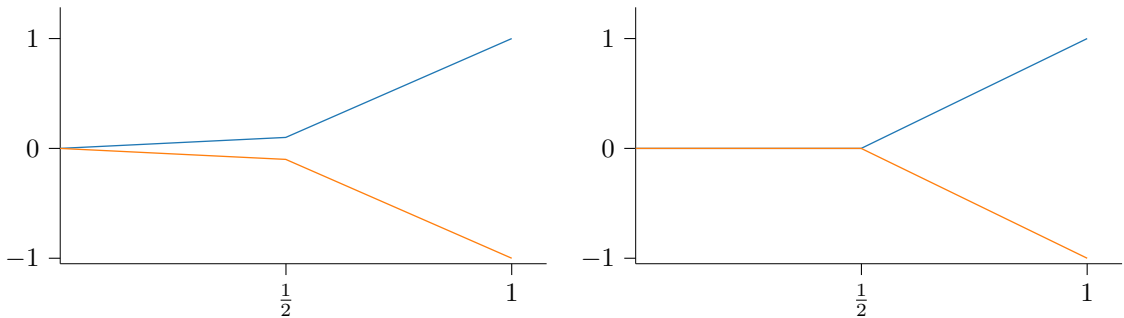

		\input{approx-example-3.pgf}
		\quad
		\input{limit-example-3.pgf}
		\caption{The two possible trajectories of $X^n$, for some $n \in \N$, are shown on the left, and the two possible trajectories of $X^\infty$ on the right.}\label{fig:motivation}
	\end{figure}
		For $n \in \N \cup \{\infty\}$, define the process $(X^n_t)_{t \in [0, 1]}$ with $X^n_0 = 0$ such that
		\begin{equation}
			\begin{split}
				X^n_{\frac{1}{2}} = \frac{1}{n}\quad\text{and} \quad X^n_1 = 1, &\quad \text{with probability}\; \frac{1}{2},\\
				X^n_{\frac{1}{2}} = -\frac{1}{n} \quad \text{and} \quad X^n_1 = -1, &\quad \text{with probability}\; \frac{1}{2},
			\end{split}
		\end{equation}
		linearly interpolated for intermediate times,
		and define $\mu^n\coloneqq  \Law(X^n)$. The two possible trajectories of $X^n$, $n \in \N$ are shown on the left-hand side of \Cref{fig:motivation}, and the trajectories of $X^\infty$ are shown on the right-hand side of \Cref{fig:motivation}.
		
		One can see that the behaviour of the approximating processes after time $1/2$ is completely determined by the history of the process up to that time, whereas the behaviour of the limiting process after time $1/2$ is independent of the past. The classical Wasserstein distance cannot distinguish these differing information structures.
		
		For each $n \in \N$, it is possible to couple $\mu^n$ and $\mu^\infty$ in such a way that paths that terminate at a positive value are mapped onto each other, and likewise for negative values. Thus we can see that the Wasserstein distance between $\mu^n$ and $\mu^\infty$ converges to $0$ as $n \to \infty$. Such a coupling is not bicausal, however. In fact the only bicausal coupling is the product coupling, which maps paths that terminate at a positive value onto those that terminate at a negative value with probability $1/2$. We can thus bound the adapted Wasserstein distance $\AW_p(\mu^n, \mu^\infty)$ from below by a positive constant, for any $p \geq 1$.
		Note that if we take the right-continuous version of the filtration in the definition of causality, \Cref{def:bicausal_couplings}, then the previous argument still holds.
		
		Consider finally the problem of finding $V^n\coloneqq  \inf_{Z \;\F^{X^n}_{1/2}\text{-measurable}}\E|X^n_1 - Z|^2$, for each $n \in \N \cup \{\infty\}$, with $\F^{X^n}$ denoting the raw natural filtration of $X^n$. Since $X^n_1$ is $\F^{X^n}_{1/2}$-measurable, for each $n \in \N$, we have $V^n = 0$. On the other hand, $V^\infty = 1$, since the sigma-algebra $\F^{X^\infty}_{1/2}$ is trivial. Thus we see that $V^n$ does not converge to $V^\infty$ as $n \to \infty$. This exemplifies how the classical Wasserstein distance fails to capture the role of information in dynamic decision problems.
	\end{example}
	
	\subsection{Non-Markovianity}

	We now show that, if the coefficients of the SDEs \eqref{eq:sdes} are non-Markovian, then the synchronous coupling may fail to attain the adapted Wasserstein distance between the laws of the solutions of \eqref{eq:sdes}.

	\begin{example}[non-Markovian counterexample]\label{ex:non-M}
	This example already appears in \cite{BaBaBeEd19a}, as a counter-example in a different setting.
	 
		Let $C>0$, $h\in(0,1)$, and define $b(t, \omega) \coloneqq  C \, \text{sign}(\omega_h)\mathds{1}_{\{t>h\}}$, for $t \in [0, 1]$, $\omega \in \Omega$. Then let $\mu \coloneqq \Law(X)$, where $X$ is the unique strong solution of
		\begin{equation}
		\di X_t=b(t, X)\dt+\di W_t,
		\quad X_0=0;
	\end{equation}
	note that strong existence and pathwise uniqueness are guaranteed by Zvonkin \cite{Zv74}. Similarly, let $\nu \coloneqq \Law(\bar X)$, where $\bar X$ is the unique strong solution of
		\begin{equation}
		\di \bar X_t=-b(t, \bar X)\dt+\di W_t,
		\quad \bar X_0=0.
	\end{equation}
	Now consider the couplings
	\begin{align}
		\pi^\sync& \coloneqq \Law\!\left (W_\cdot + C\, \text{sign}(W_h) [\cdot-h]_+\,\,\, , \,\,\,  W_\cdot- C\, \text{sign}(W_h) [\cdot-h]_+ \right ); \\
		\pi^\async& \coloneqq \Law\!\left (W_\cdot + C\, \text{sign}(W_h) [\cdot-h]_+\,\,\, , \,\,\,  -W_\cdot + C\, \text{sign}(W_h) [\cdot-h]_+ \right ).
	\end{align}
	Note that $\pi^\sync, \pi^\async \in \cplba(\mu, \nu)$. In fact, $\pi^\sync$ is the synchronous coupling between its marginals, whereas arguably the \emph{anti-synchronous} coupling $\pi^\async$ is the opposite of the synchronous coupling (cf.\ the relationship between the monotone and antitone couplings between measures on $\R$). A few computations reveal that, for the quadratic cost, we have
	\begin{align}
		\E^{\pi^\sync}\!\left[\int_0^1 | \omega_t - \bar \omega_t|^2 \D t\right]&=\E\!\left[\int_h^1(2C\, \text{sign}(W_h)[t-h])^2\di t\right] = \frac{4}{3}C(1-h)^3, \quad \text{while}\\
		\E^{\pi^\async}\!\left[\int_0^1 | \omega_t - \bar \omega_t|^2 \D t\right]&=\E\!\left[\int_0^1(2W_t)^2\di t\right] = 2.
	\end{align}
	Choosing $h$ sufficiently small and $C$ sufficiently large, we have that the expected cost of the antisynchronous coupling is strictly less than the expected cost of the synchronous coupling. Hence, the synchronous coupling does not attain the adapted Wasserstein distance $\AW_2(\mu, \nu)$.
	\end{example}

	As a counterpoint to the previous example, we highlight that we may still derive optimality of the synchronous coupling in certain non-Markovian cases.
	
	\begin{example}[kinetic SDEs]\label{ex:kinetic}
	Let $b, \bar b\colon \R \to \R$ be increasing and Lipschitz continuous, and consider, for $t \in [0, 1]$, the kinetic equations
	\begin{equation}\label{eq:kinetic}
		\begin{split}
			dX_t& = b\!\left( \int_0^t X_s\di s \right)\di t+\di W_t, \quad X_0=0; \\
			d\bar X_t &= \bar b\!\left( \int_0^t \bar X_s\di s \right)\di t+\di \bar W_t, \quad \bar X_0=0.
		\end{split}
	\end{equation}
	
	A reasonable time-discretisation of such SDEs would be to define, for $N \in \N$ and $h = \frac{1}{N}$, the process $(X^h_k)_{k = 0, \dotsc, N}$ by $X^h_0=0$ and for $k = 1, \dotsc, N$,
\[
X^h_{k} = X^h_{k-1} + b\bigg( h\sum_{i< k} X^h_i\bigg)h + W_{kh} - W_{(k-1)h}.
\]
This scheme is stochastically monotone in the sense of \Cref{def:stochastic_mon}. Defining $\bar X^h$ in the same way, \Cref{prop:kr-discrete-optimality} then gives that the Knothe--Rosenblatt rearrangement is optimal among bicausal couplings between the laws of $X^h$ and $\bar X^h$ (for the same class of cost functions).
	Similarly to \Cref{lem:kr-joint-law}, we also have that the Knothe--Rosenblatt rearrangement is given by $\Law(X^h, \bar X^h)$ when the discretisation schemes $X^h$ and $\bar X^h$ are defined with respect to a common Wiener process $W = \bar W$.
	
	By use of arguments similar to those employed for the Markovian case in \Cref{sec:proofs}, we expect to obtain appropriate convergence of the scheme to the true solution. Optimality of the synchronous coupling for the continuous-time bicausal transport problem between the laws of solutions to \eqref{eq:kinetic} would then follow from \Cref{prop:stability_optimiser}. However, we do not carry out this analysis rigorously. 
	\end{example}

	\subsection{Multidimensional examples}
	
		As stated in \Cref{rem:monotone-optimality}, our proof of optimality of the synchronous coupling does not extend to higher dimensions. In this section we present multidimensional examples for which the synchronous coupling is not optimal.
		In particular, these examples illustrate that, in higher dimensions, Markovianity is no longer an indicator of optimality of the synchronous coupling. 
		The kinetic system \eqref{eq:kinetic} in \Cref{ex:kinetic} can be made Markovian if we enlarge the state space to $\R^2$, by introducing the additional state variables $V_t=\int_0^t X_s\di s$ and $\bar V_t=\int_0^t \bar X_s\di s$, $t \in [0, 1]$. 
		In a similar fashion, \Cref{ex:non-M} can be made Markovian by introducing an additional state variable.
		
		\begin{example}[two-dimensional counterexample I]\label{ex:2d-1}
			Let $A > 0$ be small and $C > 0$ large, let $B, W$ be independent Wiener processes and consider the SDEs
			\begin{equation}
				\begin{split}
					\D X_t & = C \!\left(\mathds{1}_{\{Y_t > - A\}} - \mathds{1}_{\{Y_t < A\}}\right)\D t + \D B_t, \quad X_0 = 0,\\
					\D Y_t & = \mathds{1}_{\{|Y_t|< A\}}\D W_t, \quad Y_0 = 0.
				\end{split}
			\end{equation}
			Also let $\bar W, \bar B$ be independent Wiener processes and consider the SDEs
			\begin{equation}
				\begin{split}
					\D \bar X_t & = - C \!\left(\mathds{1}_{\{\bar Y_t > - A\}} - \mathds{1}_{\{ \bar Y_t < A\}}\right)\D t + \D \bar B_t, \quad \bar X_0 = 0,\\
					\D \bar Y_t & = \mathds{1}_{\{|\bar Y_t|< A\}}\D \bar W_t, \quad \bar Y_0 = 0.
				\end{split}
			\end{equation}
			The process $Y$ (resp.\ $\bar Y$) is a Wiener process until hitting $-A$ or $A$, where it freezes. The process $X$ (resp.\ $\bar X$) is also a Wiener process until the aforementioned hitting time, after which a drift $C$ or $-C$ is added depending on whether $Y$ had hit $A$ or $-A$ (resp. $\bar Y$ had hit $-A$ or $A$).
			
			Suppose that $\bar B = B$ and $\bar W = W$. We call this the synchronous coupling in dimension two. Then the discrepancy $\int_0^1|X_t - \bar X_t|^2\D t$ is very large, since $X$ and $\bar X$ have large drifts in different directions after the hitting time.
			
			On the other hand, take $\bar B = B$ and $\bar W = - W$. This creates a small discrepancy $\int_0^1 |Y_t - \bar Y_t|^2 \D t$, but now $X_t = \bar X_t$, for all $t \in [0, 1]$. Hence, for $A$ sufficiently small and $C$ sufficiently large, this second coupling has a lower $L^2$ cost than the synchronous coupling.
		\end{example}

		In the previous example, we considered two-dimensional systems driven by two-dimensional Wiener processes. A natural question is whether a counterexample to the optimality of the synchronous coupling can be constructed when the driving Wiener processes are one dimensional. The next example addresses this question.

		\begin{example}[two-dimensional counterexample II]\label{ex:2d-2}
			Let $W, \bar W$ be one-dimensional standard Wiener processes and let $h \in (0, 1)$, $C \in (0, \infty)$. Let $(X, Y)$ be strong solutions of the SDEs
			\begin{equation}
				\begin{split}
					\D X_t & = C \sign(Y_t) \ind{t > h} \D t + \D W_t, \\
					\D Y_t & = \ind{t \leq h} \D W_t,
				\end{split}
			\end{equation} 
			with $X_0 = Y_0 = 0$, and write $\mu = \Law(X, Y)$.
			Also let $(\bar X, \bar Y)$ be strong solutions of the SDEs
			\begin{equation}
				\begin{split}
					\D \bar X_t & = - C \sign(\bar Y_t) \ind{t > h} \D t + \D \bar W_t,\\
					\D \bar Y_t & = \ind{t \leq h} \D \bar W_t,
				\end{split}
			\end{equation}
			with $\bar X_0 = \bar Y_0$, and write $\nu = \Law(\bar X, \bar Y)$. Then we can again define the synchronous coupling $\pi^\sync_{\mu, \nu} \coloneqq \Law(X, Y, \bar X, \bar Y)$ when $W = \bar W$, and the antisynchronous coupling $\pi^\async_{\mu, \nu} \coloneqq \Law(X, Y, \bar X, \bar Y)$ when $W = - \bar W$.
			
			Observe that, for any $t \in [0, 1]$, we have $Y_t = W_{t \wedge h}$ and $\bar Y_t = \bar W_{t \wedge h}$, and so
			\begin{equation}
				\begin{split}
					\D X_t & = C \sign(W_h) \ind{t > h} \D t + \D W_t,\\
					\D \bar X_t & = - C \sign(\bar W_h) \ind{t > h} \D t + \D \bar W_t,
				\end{split}
			\end{equation}
			as in \Cref{ex:non-M}.
			The squared $L^2$-cost under the synchronous coupling is now calculated to be
			\begin{equation}
				\E^{\pi^\sync_{\mu, \nu}}\!\left[\int_0^1|X_t - \bar X_t|^2 \D t + \int_0^1 |Y_t - \bar Y_t|^2 \D t\right] = \E^{\pi^\sync_{\mu, \nu}}\!\left[\int_0^1|X_t - \bar X_t|^2 \D t \right] = \frac{4}{3} C (1 - h)^3.
			\end{equation}
			And for the antisynchronous coupling, the cost is given by
			\begin{equation}
				\begin{split}
					\E^{\pi^\async_{\mu, \nu}}\!\left[\int_0^1|X_t - \bar X_t|^2 \D t + \int_0^1 |Y_t - \bar Y_t|^2 \D t\right] & = \E^{\pi^\async_{\mu, \nu}}\!\left[\int_0^1|X_t - \bar X_t|^2 \D t \right] + 4\int_0^1 \E\!\left[|W_{t \wedge h}|^2\right] \D t\\
					& = 2 + 4 \int_0^1(t \wedge h)\D t \leq 2 + 4h.
				\end{split}
			\end{equation}
			Taking $C$ sufficiently large and $h$ sufficiently small, we conclude that the synchronous coupling may fail to be optimal.
		\end{example}

	\subsection{Adapted Wasserstein distance with $L^\infty$ norm}
		Suppose that we wish to replace the $L^p$ norm on $\Omega = C([0, 1], \R)$ with the $L^\infty$ norm and find the associated adapted $1$-Wasserstein distance between measures $\mu, \nu$ on $\Omega$:
		\begin{equation}\label{eq:sup-bcot}
			\inf_{\pi \in \cplba(\mu, \nu)}\E^\pi\!\left[\sup_{t \in [0, 1]}|\omega_t - \bar \omega_t|\right].
		\end{equation}
		Discretising the problem as before, we arrive at
		\begin{equation}\label{eq:max-bcot}
			\inf_{\pi \in \cplba(\mu^N, \nu^N)}\E^\pi\!\left[\max_{k \in \{1, \dotsc, N\}}|x_k - \bar x_k|\right],
		\end{equation}
		for some $N \in \N$ and measures $\mu_N, \nu_N$ on $\R^N$. This discrete-time bicausal optimal transport problem does not satisfy the assumptions of \Cref{prop:kr-discrete-optimality} for optimality of the Knothe--Rosenblatt rearrangement, since the cost function is not of a separable form.
		
		We now give a counterexample to optimality of the Knothe--Rosenblatt rearrangement for the problem \eqref{eq:max-bcot}, when $N = 2$ and the marginals are atomic. We leave open the question of finding the optimiser when the marginals are the laws of some numerical scheme for an SDE.
		
		\begin{example}
			Define the measures $\mu \coloneqq 1/2 \cdot (\delta_{(-3, -7)} + \delta_{(1, 4)})$ and $\nu \coloneqq 1/2 \cdot (\delta_{(2, 4)} + \delta_{(5, 6)})$ on $\R^2$. We will show that the Knothe--Rosenblatt rearrangement is suboptimal for
			\begin{equation}
				\inf_{\pi \in \cplba(\mu, \nu)}\E^\pi\!\left[|x_1 - \bar x_1| \vee |x_2 - \bar x_2|\right].
			\end{equation}
			We illustrate processes with law $\mu$ and $\nu$, respectively, along with two bicausal couplings in \Cref{fig:counterexample}.
			
			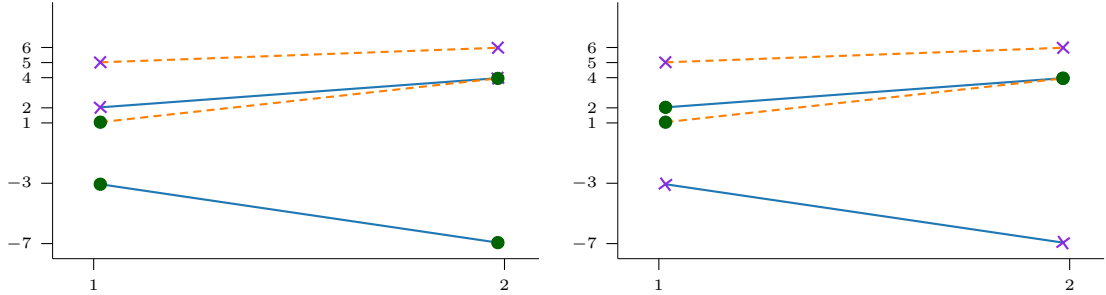
\begin{figure}[h]
				\begin{tikzpicture}

\definecolor{darkgray176}{RGB}{176,176,176}
\definecolor{darkorange25512714}{RGB}{255,127,14}
\definecolor{steelblue31119180}{RGB}{31,119,180}
\definecolor{darkgreen}{RGB}{0, 100, 0}
\definecolor{blueviolet}{RGB}{138,43,226}

\begin{axis}[
tick align=outside,
tick pos=left,
x grid style={darkgray176},
xmin=0.9, xmax=2.01,
xtick style={color=black},
xtick = {1.0, 2.0},
xticklabels = {\tiny{$1$}, \tiny{$2$}},
y grid style={darkgray176},
ymin=-8, ymax=7,
ytick style={color=black},
ytick = {-7.0, -3.0, 1.0, 2.0, 4.0, 5.0, 6.0},
yticklabels = {\tiny{$-7$}, \tiny{$-3$}, \tiny{$1$}, \tiny{$2$}, \tiny{$4$}, \tiny{$5$}, \tiny{$6$}}
]
\draw[{Rays[n = 20, color = darkgreen, scale = 0.8]}-{Rays[n = 20, color = darkgreen, scale = 0.8]}, steelblue31119180, thick](1.0, -3.0) -- (2.0, -7.0);
\draw[{Rays[color = blueviolet]}-{Rays[color = blueviolet]}, steelblue31119180, thick](1.0, 2.0) -- (2.0, 4.0);
\draw[{Rays[n = 20, color = darkgreen, scale = 0.8]}-{Rays[n = 20, color = darkgreen, scale = 0.8]}, densely dashed, orange, thick](1.0, 1.0) -- (2.0, 4.0);
\draw[{Rays[color = blueviolet]}-{Rays[color = blueviolet]}, densely dashed, orange, thick](1.0, 5.0) -- (2.0, 6.0);
\end{axis}


\end{tikzpicture}
				\quad
				\begin{tikzpicture}

\definecolor{darkgray176}{RGB}{176,176,176}
\definecolor{darkorange25512714}{RGB}{255,127,14}
\definecolor{steelblue31119180}{RGB}{31,119,180}
\definecolor{darkgreen}{RGB}{0, 100, 0}
\definecolor{blueviolet}{RGB}{138,43,226}

\begin{axis}[
tick align=outside,
tick pos=left,
x grid style={darkgray176},
xmin=0.9, xmax=2.01,
xtick style={color=black},
xtick = {1.0, 2.0},
xticklabels = {\tiny{$1$}, \tiny{$2$}},
y grid style={darkgray176},
ymin=-8, ymax=7,
ytick style={color=black},
ytick = {-7.0, -3.0, 1.0, 2.0, 4.0, 5.0, 6.0},
yticklabels = {\tiny{$-7$}, \tiny{$-3$}, \tiny{$1$}, \tiny{$2$}, \tiny{$4$}, \tiny{$5$}, \tiny{$6$}}
]
\draw[{Rays[color = blueviolet]}-{Rays[color = blueviolet]}, steelblue31119180, thick](1.0, -3.0) -- (2.0, -7.0);
\draw[{Rays[n = 20, color = darkgreen, scale = 0.8]}-{Rays[n = 20, color = darkgreen, scale = 0.8]}, steelblue31119180, thick](1.0, 2.0) -- (2.0, 4.0);
\draw[{Rays[n = 20, color = darkgreen, scale = 0.8]}-{Rays[n = 20, color = darkgreen, scale = 0.8]}, densely dashed, orange, thick](1.0, 1.0) -- (2.0, 4.0);
\draw[{Rays[color = blueviolet]}-{Rays[color = blueviolet]}, densely dashed, orange, thick](1.0, 5.0) -- (2.0, 6.0);
\end{axis}


\end{tikzpicture}
				\caption{The Knothe--Rosenblatt rearrangement $\pi^\kr_{\mu, \nu}$ is shown on the left, and the coupling $\pi^\antitone_{\mu, \nu}$ on the right. The solid blue lines and dashed orange lines represent processes with law $\mu$ and $\nu$, respectively. At each time, the points with the same colour and style are coupled with each other.}
				\label{fig:counterexample}
			\end{figure}
			For the Knothe--Rosenblatt rearrangement $\pi^\kr_{\mu, \nu}$ we have
			\begin{equation}
				\begin{split}
					\E^{\pi^\kr_{\mu, \nu}}\!\left[|x_1 - \bar x_1| \vee |x_2 - \bar x_2|\right] & = 1/2 \cdot (|- 3 - 2| \vee |- 7 - 4|) + 1/2 \cdot (|1 - 5| \vee |4 - 6|)\\
					& = 1/2 \cdot (11 + 4) = 15/2.
				\end{split}
			\end{equation}
			On the other hand, consider the coupling $\pi^\antitone_{\mu, \nu}$ that is defined similarly to $\pi^\kr_{\mu, \nu}$ but with the first marginals coupled via the \emph{antitone} rearrangement rather than the monotone rearrangement. Then
			\begin{equation}
				\begin{split}
					\E^{\pi^\antitone_{\mu, \nu}}\!\left[|x_1 - \bar x_1| \vee |x_2 - \bar x_2|\right] & = 1/2 \cdot (|- 3 - 5| \vee |- 7 - 6|) + 1/2 \cdot (|1 - 2| \vee |4 - 4|)\\
					& = 1/2 \cdot (13 + 1) = 14/2 < 15/2.
				\end{split}
			\end{equation}
			Thus the Knothe--Rosenblatt rearrangement is not an optimiser.
			
			Note that both $\mu$ and $\nu$ are stochastically increasing. Therefore \Cref{prop:kr-discrete-optimality} implies that the Knothe--Rosenblatt rearrangement $\pi^\kr_{\mu, \nu}$ is an optimiser of
			\begin{equation}
				\inf_{\pi \in \cplba(\mu, \nu)}\E^\pi\!\left[|x_1 - \bar x_1| + |x_2 - \bar x_2|\right].
			\end{equation}
			In fact one can compute that both $\pi^\kr_{\mu, \nu}$ and $\pi^\antitone_{\mu, \nu}$ attain the same value for this problem and are therefore both optimisers.
		\end{example}

\appendix

\section{Convergence of the monotone scheme}	\label{app:convergence}

The aim here is to prove \Cref{lem:convergence-truncated} which states the $L^p$-convergence of the monotone Euler--Maruyama scheme (given in \Cref{def:monotone-em}) to the unique strong solution of \eqref{eq:sde} when the coefficients are Lipschitz. 
The proof of this result proceeds as follows: we first establish $L^2$-convergence by showing that the monotone Euler--Maruyama scheme is close in the $L^2$-norm to the standard Euler--Maruyama scheme; making use of a bound on the $p$\textsuperscript{th} moments of the monotone scheme, we then deduce $L^p$-convergence. 
Throughout the following proofs, we make use of generic constants, which may change from one line to the next.

\begin{remark}\label{rem:ito-strong-solution}
	Under the assumption that the coefficients $b, \sigma$ in \eqref{eq:sde} are Lipschitz, there exists a unique strong solution $X$ to \eqref{eq:sde} according to a classical result of It\^o; see, e.g.\ \cite[Proposition 1.9]{ChEn05}. Moreover, for $p > 0$, the process $X$ satisfies the following moment bounds. There exist constants $C_p, \tilde C_p > 0$ such that
	\begin{equation}\label{eq:lp-ito}
		\E \!\left[\sup_{0 \leq t \leq 1}|X_t|^p\right] \leq C_p,
	\end{equation}
	and, for any $s, t \in [0, 1]$ with $s < t$,
	\begin{equation}\label{eq:one-step-error-sde}
		\E \!\left[\sup_{s \leq u \leq t}|X_u - X_s|^p\right] \leq \tilde C_p (t - s)^{\frac p2}.
	\end{equation}
	These bounds follow from a standard application of Doob's martingale inequality, the Burkholder--Davis--Gundy inequality, and Gr\"onwall's lemma; see, e.g.~\cite[Lemma 3.8 and Equation~(3.48)]{GiSk79}.
\end{remark}

\begin{remark}\label{rem:truncated-bm-martingale}
	Note that the process $(W^h_t)_{t \in [0, 1]}$ is a martingale with respect to $\mathcal{F}^W$, the filtration generated by the Wiener process $W$ augmented to satisfy the usual conditions.
\end{remark}

We start with a lemma, which we adapt from \cite[Lemma 2.1]{MiReTr02}, that gives a bound on the fourth moment of the error created by the truncation. This bound is used in the proof of \Cref{prop:truncated-euler-close-to-euler} and thus justifies the choice of the truncation level $A_h$.

\begin{lemma}\label{lem:truncated-BM}
	For $N \in \N$, $h = 1/N$, and fixed $k \in \{0, \dotsc, N-1\}$,  we have the second moment bound
	\begin{equation}
		\E \!\left \lvert W_{{(k+1)h}} - W_{{kh}} - (W^h_{{(k+1)h}} - W^h_{{kh}}) \right \rvert^2 \le 2 h^{3}.
	\end{equation}
\end{lemma}

{\begin{proof}
	First note that
	\begin{equation}
		\begin{split}
			\E \!\left \lvert W_{{(k+1)h}} - W_{{kh}} - (W^h_{{(k+1)h}} - W^h_{{kh}}) \right \rvert^2 & = \E \!\left|W_{(k + 1)h} - W_{(k + 1)h \wedge \tau^h_k}\right|^2\\
			& = \E [h -  h \wedge \tau^h_0] \leq h \P [\tau^h_0 \leq h],
		\end{split}
	\end{equation}
	since $W$ has identically distributed increments, and calculate
	\begin{equation}
		\P[\tau^h_0 \le h] = 2\P\!\left[\sup_{t \in[0, h]}W_t \ge A_h \right] = 4 \P [W_h \geq A_h],
	\end{equation}
	using the reflection principle.
	Then
	\begin{equation}
		\begin{split}
			\E \!\left \lvert W_{{(k+1)h}} - W_{{kh}} - (W^h_{{(k+1)h}} - W^h_{{kh}}) \right \rvert^2 & \leq \frac{4 h}{\sqrt{2 \pi h}}\int_0^\infty e^{- \frac{(x + A_h)^2}{2h}}\D x < 2h e^{-\frac{A_h^2}{2h}}.
		\end{split}
	\end{equation}
	Recalling the definition $A_h = 2 \sqrt{-h \log h}$, we conclude.
\end{proof}
}

\begin{remark}\label{rem:moment-bound}
	From the above proof, we see that, for an arbitrary $K \in \N$, we can redefine $A_h\coloneqq  K \sqrt{-h\log h}$ and achieve a second moment bound of $2 h^{1 + \frac{K^2}{2}}$ in \Cref{lem:truncated-BM}.
\end{remark}

In order to prove the $L^2$-convergence of the monotone Euler--Maruyama scheme $X^h$ to the unique strong solution $X$ of \eqref{eq:sde}, we first recall the following estimates for the standard Euler--Maruyama scheme $\tilde{X}^h$ (defined in \eqref{eq:em}) when the coefficients are Lipschitz.
From the proof of \cite[Theorem~10.2.2]{KlPl92}, for example, there exists a constant $\tilde{C}_0$ such that, for any $h > 0$, we have the $L^2$ estimate 
\begin{equation}\label{eq:euler-convergence}
	\E \!\left[ \sup_{0 \leq s \leq 1} \lvert \tilde{X}^h_s - X_s \rvert^2 \right] \leq \tilde{C}_0 h.
\end{equation}
Similarly to \Cref{rem:ito-strong-solution}, one can also derive the following moment bound. For any $p \geq 1$, there exists a constant $\tilde{C}_p$ such that, for any $h > 0$,
\begin{equation}\label{eq:euler-l2-bound}
	\E \!\left[ \sup_{0 \leq s \leq 1} \lvert \tilde{X}^h_s\rvert^p\right] \leq \tilde{C}_p.
\end{equation}

\begin{proposition}\label{prop:truncated-euler-close-to-euler}
	Suppose that the coefficients $b$ and $\sigma$ in \eqref{eq:sde} are Lipschitz. Then there exists a constant $C > 0$ such that, for any $N \in \N$ and $h = 1/N$,
	\begin{equation}
		\E \!\left[\sup_{0 \leq s \leq 1} \lvert \tilde{X}^h_s - X^h_s\rvert^2\right] \leq C h.
	\end{equation}
\end{proposition}

\begin{proof}
	Fix $N \in \N$ and $h = 1/N$. For $s \in [0, 1]$, introduce the notation $t_{n_s}\coloneqq  \sup\{t \leq s \colon t = kh, \; \text{for some}\; k = 0, \dotsc, N\}$ and, for $t \in [0, 1]$, define the remainder terms
	\begin{equation}
	\begin{split}
		R_t &\coloneqq  \E\!\left[\sup_{0 \leq s \leq t} \!\left \lvert \int_{0}^s \!\left(b(\tilde{X}^h_{t_{n_r}}) - b(X^h_{t_{n_r}})\right) \D r \right \rvert^2\right], \quad S_t\coloneqq  \E \!\left[\sup_{0 \leq s \leq t} \!\left \lvert \int_{0}^s \!\left(\sigma(\tilde{X}^h_{t_{n_r}}) - \sigma(X^h_{t_{n_r}})\right)\D W^h_r \right \rvert^2\right],\\
		U_t &\coloneqq  \E \!\left[\sup_{0 \leq s \leq t} \!\left \lvert \int_{0}^s \sigma(\tilde{X}^h_{t_{n_r}})\D W_r - \int_{0}^s\sigma(\tilde{X}^h_{t_{n_r}})\D W^h_r \right \rvert^2\right],
	\end{split}
	\end{equation}
	so that
	\begin{equation}
		Z_t\coloneqq  \E\!\left[\sup_{0 \leq s \leq t} \lvert \tilde{X}^h_s - X^h_s \rvert^2 \right] \leq C(R_t + S_t + U_t).
	\end{equation}	
	
	Fix $t \in [0, 1]$. By Jensen's inequality, we can bound
	\begin{equation}
		R_t \leq \int_0^t \E \!\left[ \sup_{0 \leq s \leq u} \!\left\lvert b(\tilde{X}^h_{t_{n_s}}) - b(X^h_{t_{n_s}}) \right\rvert\!^2\right] \D u.
	\end{equation}
	Then, using the Lipschitz property of $b$ and expanding the set of times over which we take the supremum, we can find a constant $C^R$ such that
	\begin{equation}
		R_t \leq C^R \int_0^t \E\!\left[\sup_{0 \leq s \leq u} \lvert \tilde{X}^h_s - X^h_s  \rvert^2 \right]\D u = C^R \int_0^t Z_u \D u.
	\end{equation}
	
	As noted in \Cref{rem:truncated-bm-martingale}, $W^h$ is an $\F^W$-martingale, and we see that $\D \langle W^h\rangle_t \leq \D t$. Thus, by Doob's martingale inequality,
	\begin{equation}
	\begin{split}
		S_t & \leq 4 \int_0^t \E \!\left[ \sup_{0 \leq s \leq u} \lvert \sigma(\tilde{X}^h_{t_{n_s}}) - \sigma(X^h_{t_{n_s}}) \rvert^2\right] \D u.
	\end{split}
	\end{equation}
	In the same way as for $R_t$, we now use the Lipschitz property of $\sigma$ to find a constant $C^S$ such that $S_t \leq C^S \int_0^t Z_u \D u$.
	
	Finally, we bound the term $U_t$. For each $k = 0, \dotsc, N - 1$, let us write $\Delta_{k+1} W = W_{(k+1)h} - W_{kh}$ and $\Delta_{k+1} W^h = W^h_{(k+1)h} - W^h_{kh}$. Then, applying Doob's inequality and the independence of increments, we get
	\begin{equation}
		\begin{split}
			U_t & \leq 4 \E \!\left[\bigg\lvert \sum_{k = 0}^{n_t - 1} \sigma(\tilde{X}^h_{{kh}}) \!\left[\Delta_{k+1} W - \Delta_{k+1} W^h \right] + \sigma(\tilde{X}^h_{t_{n_t}}) \!\left[W_t - W_{t_{n_t}} - (W^h_t - W^h_{t_{n_t}})\right]\bigg \rvert^2 \right]\\
			& \leq 4N\sum_{k = 0}^{N-1} \E \!\left[\sigma(\tilde{X}^h_{kh})^2(\Delta_{k+1} W - \Delta_{k+1} W^h)^2\right] = 4 N \sum_{k = 0}^{N - 1}\E \!\left[\sigma(\tilde{X}^h_{kh})^2\right]\E \!\left[(\Delta_{k+1} W - \Delta_{k+1} W^h)^2\right].
		\end{split}
	\end{equation}
	Applying \Cref{lem:truncated-BM}, we can bound the term $\E[(\Delta W_{k+1} - \Delta W^h_{k+1})^2] \le 2h^{3}$, for each $k = 0, \dotsc, N - 1$. Using the Lipschitz property of $\sigma$ and the $L^p$ bound \eqref{eq:euler-l2-bound} for the Euler--Maruyama scheme, we can also bound $\E[\sigma(\tilde{X}^h_{kh})^2] \leq C$, for each $k = 0, \dotsc, N - 1$. Therefore we have
	\begin{equation}
		U_t \leq \bar{C} N^2 h^3 = \bar{C} h.
	\end{equation}
	
	Combining the bounds on $R_t$, $S_t$ and $U_t$, and defining $C = C^R + C^S$, we can bound $Z_t$ by
	\begin{equation}
		Z_t \leq \bar{C}h + C \int_0^t Z_u \D u,
	\end{equation}
	and by Gr\"onwall's inequality we conclude that $Z_t \leq \tilde{C} h$, for some constant $\tilde C >0$.
\end{proof}

\begin{remark}
	Similarly to \Cref{rem:moment-bound}, the power in the bound in \Cref{prop:truncated-euler-close-to-euler} can be made arbitrarily large, by multiplying the truncation level $A_h$ by a sufficiently large constant.
\end{remark}

The following immediate corollary now gives a rate for the $L^2$-convergence of the monotone Euler--Maruyama scheme $X^h$ to the solution $X$ of the SDE \eqref{eq:sde}.

\begin{corollary}\label{cor:l2-convergence}
	Suppose that the coefficients $b$ and $\sigma$ in \eqref{eq:sde} are Lipschitz. Then there exists a constant $C > 0$ such that, for any $h > 0$ sufficiently small,
	\begin{equation}
		\E \!\left[ \sup_{0 \leq s \leq 1} \lvert X^h_s - X_s \rvert^2 \right] \leq C h.
	\end{equation}
\end{corollary}
\begin{proof}
	Combining the rate of $L^2$ convergence of the Euler--Maruyama scheme given in \eqref{eq:euler-convergence} with the estimate of the $L^2$-error between the Euler--Maruyama scheme and the monotone Euler--Maruyama scheme given in \Cref{prop:truncated-euler-close-to-euler}, we can conclude via a simple application of the triangle inequality that
	\begin{equation}
		\E \!\left[\sup_{0 \leq s \leq 1} \lvert X_s - X^h_s\rvert^2\right] \leq 2 \E \!\left[ \sup_{0 \leq s \leq 1} \lvert \tilde{X}^h_s - X_s \rvert^2 \right] + 2 \E \!\left[\sup_{0 \leq s \leq 1} \lvert \tilde{X}^h_s - X^h_s\rvert^2\right] \leq C h.
	\end{equation}
\end{proof}

In order to obtain $L^p$-convergence, we make use of the following bounds on the $p$\textsuperscript{th} moments of the monotone Euler--Maruyama scheme $X^h$, for $h > 0$.

\begin{lemma}\label{lem:truncated-scheme-l2-bound}
	Suppose that the coefficients $b$ and $\sigma$ in \eqref{eq:sde} are Lipschitz. Then, for $p\ge 1$, there exists a constant $C_p > 0$, depending only on the initial condition $x_0$ and the Lipschitz constants of the coefficients $b$ and $\sigma$, such that for any $h>0$,
	\begin{equation}
		\E \!\left[\sup_{0 \leq t \leq 1}\!\left\lvert X^h_t \right\rvert\!^p\right] \leq C_p.
	\end{equation}
\end{lemma}

\begin{proof}
	Follows from a standard application of martingale inequalities and Gr\"onwall's lemma, similarly to \Cref{rem:ito-strong-solution}.
\end{proof}

Convergence in $L^p$ now follows immediately.

\begin{proof}[Proof of \Cref{lem:convergence-truncated}]
	By \Cref{cor:l2-convergence}, $X^h$ converges to $X$ in $L^2$, and hence in $L^q$ for all $q \in [1, 2]$. For fixed $p \geq 2$, \Cref{lem:truncated-scheme-l2-bound} gives a bound on the $(p+1)$\textsuperscript{th} moment of $X^h$. Moreover, the $(p+1)$\textsuperscript{th} moment of $X$ is bounded by \eqref{eq:lp-ito}. Combining the $L^2$-convergence with the bounds in $L^{p+1}$ implies $L^p$-convergence, as required.

To prove $L^p$-convergence of the discrete-time scheme, note that
\begin{equation}
	\E \!\left[\sum_{k = 0}^{N - 1} \int_{kh}^{(k+1)h}\lvert X_{kh} - X^h_{kh} \rvert^p \D t\right] \leq T \E \!\left[\sup_{0 \leq t \leq 1}|X_t - X^h_t|^q\right] \xrightarrow{h \to 0} 0,
\end{equation}
and \eqref{eq:one-step-error-sde} provides the estimate
\begin{equation}
	\E \!\left[\sum_{k = 0}^{N - 1} \int_{kh}^{(k+1)h}\lvert X_t - X_{kh} \rvert^p \D t\right] \leq h\sum_{k = 0}^{N - 1}\E \!\left[\sup_{kh \leq t \leq (k + 1)h}|X_t - X_{kh}|^p\right] \leq \tilde C_p T h^{\frac p2}.
\end{equation}
We conclude by the triangle inequality.
\end{proof}

\section{A stability result for SDEs}

We here establish a stability result for the parameter dependence of path-dependent SDEs driven by correlated Wiener processes. The result can also be obtained from \cite[Theorem~3.24]{JaMe81}, for example, but for completeness we provide a direct proof.

In order to formulate the result, recall the notation $\Omega=C([0,1],\R)$ and $\|\omega\|_\infty \coloneqq  \sup_{s\in[0,1]}|\omega_s|$, $\omega\in \Omega$, for the sup-norm. As before, $\Omega$ is equipped with the canonical filtration and the uniform topology. 
We also equip $\Omega \times \Omega$ etc.\ with the product filtration and product topology. Moreover, for $p \geq 1$, and $\pi,\pi^\prime\in \Pc(\Omega \times \Omega)$ with finite $p$\textsuperscript{th} moment, we here define the $p$-Wasserstein distance (with respect to the sup-norm) between $\pi$ and $\pi^\prime$ to be 
\begin{equation}\label{eq:wass-stab}
	\inf_{\alpha\in \cpl(\pi,\pi^\prime)}\E^\alpha \!\left[\|\omega - \omega^\prime\|_\infty^p + \|\bar\omega - \bar \omega^\prime\|_\infty^p\right],
\end{equation}
where $\cpl(\pi,\pi^\prime)$ denotes the set of probability measures on $(\Omega \times \Omega)\times( \Omega \times \Omega)$ with marginal distribution onto the first (resp. last) two coordinates given by $\pi$ (resp. $\pi^\prime$) and $((\omega,\bar\omega),(\omega^\prime,\bar\omega^\prime))$ denotes the canonical process.

\begin{remark}\label{rem:wass-p}
	For any $p \geq 1$, $\pi_n$ converges to $\pi$ with respect to the $p$-Wasserstein distance on $\Pc(\Omega \times \Omega)$, as defined in \eqref{eq:wass-stab}, if and only if, for any continuous function $\phi\colon \Omega \times \Omega \to \R$ with at most polynomial growth of order $p$ --- i.e.\ $|\phi(\omega, \bar\omega)| \leq C (1 + \|\omega\|_\infty^p + \|\bar\omega\|_\infty^p)$, $(\omega, \bar\omega)\in \Omega \times \Omega$ --- it holds that $\E^{\pi_n}[\phi(\omega, \bar\omega)] \to \E^\pi[\phi(\omega, \bar\omega)]$ (see, e.g.~\cite[Theorem 7.12]{Vi03}).
\end{remark}

\begin{proposition}\label{prop:stab_abstract}
Let $(W, \bar W)$ be a $\rho$-correlated Wiener process, for some progressively measurable process $\rho$, as defined in \Cref{def:correlated}. 
Suppose that $(x_0, \bar x_0, b, \bar{b}, \sigma, \bar{\sigma})$ satisfies \Cref{ass:stability}, and write $(X, \bar X)$ for the unique strong solution of \eqref{eq:sde-path-dep} driven by $(W, \bar W)$.

For $n \in \N$, consider also $(x^n_0, \bar x^n_0, b^n,\bar b^n,\sigma^n,\bar\sigma^n)$ satisfying \Cref{ass:stability}.(i) and (iii), with a uniform slope constant $K$ in \eqref{eq:linear-growth}, and such that strong existence holds for \eqref{eq:sde-path-dep}; let $(X^{b^n, \sigma^n},\bar{X}^{\bar{b}^n, \bar{\sigma}^n})$ be one such strong solution, when \eqref{eq:sde-path-dep} is driven by $(W, \bar W)$.

Suppose also that, as $n \to \infty$, $(x^n_0,\bar x^n_0)\to(x_0,\bar x_0)$ and the following convergence holds:
\begin{align}\label{eq:unifconvassump_app}
\|\omega^n - \omega \|_\infty \to 0 \implies (b^n,\bar b^n,\sigma^n,\bar\sigma^n)(t,\omega^n)\to (b,\bar b,\sigma,\bar\sigma)(t,\omega), \text{ for each }t \in [0, 1].
\end{align}

Then, for any $p \geq 1$,
\begin{equation}
	\Law(X^{b^n, \sigma^n},\bar{X}^{\bar{b}^n, \bar{\sigma}^n}) \xrightarrow{n \to \infty} \Law(X,\bar X),
\end{equation}
in the $p$-Wasserstein distance (with respect to the sup-norm) on $\Pc(\Omega \times \Omega)$.
\end{proposition}

\begin{proof}
Similarly to \Cref{rem:ito-strong-solution}, standard SDE estimates based on the BDG inequality, Jensen's inequality, and Gr\"onwall's lemma show the existence of $K_p<\infty$ such that $\E[\|X^{b^n, \sigma^n}\|_\infty^p]\leq K_p(1+|x^n_0|^p) $, with similar bounds for $\bar{X}^{\bar{b}^n, \bar{\sigma}^n}$. As $(x^n_0,\bar x^n_0)_{n \in \N}$ converges, this shows that, for all $p \ge 1$, the $p$\textsuperscript{th} moments of $\|X^{b^n, \sigma^n}\|_\infty$ and $\|\bar{X}^{\bar{b}^n, \bar{\sigma}^n}\|_\infty$ are bounded uniformly in $n\in \N$. On the one hand, this implies that $(\Law(X^{b^n, \sigma^n},\bar{X}^{\bar{b}^n, \bar{\sigma}^n}))_{n \in \N}$ is tight, and on the other hand that it suffices to prove that $\Law(X^{b^n, \sigma^n},\bar{X}^{\bar{b}^n, \bar{\sigma}^n}) \to \Law(X,\bar X)$ weakly on $\Pc(\Omega \times \Omega)$. Thanks to \Cref{ass:stability}.(iv), this can be achieved if we prove that each weak accumulation point of $(\Law(X^{b^n, \sigma^n},\bar{X}^{\bar{b}^n, \bar{\sigma}^n}))_{n \in \N}$ solves the martingale problem associated to the system for $(X,\bar X)$, since the latter is then well posed.

Let $\eta \in \Pc(\Omega \times \Omega)$ be one such weak accumulation point. Then, after possibly passing to a subsequence, Skorokhod's representation theorem ensures the existence of stochastic processes $(X^n,\bar X^n,W^n,\bar W^n)_{n \in \N}$ and $(Y, \bar Y, W^\infty, \bar W^\infty)$ defined on a single probability space $(\tilde{\Omega}, \tilde{\mathcal F}, \tilde{\mathbb P})$ such that
	\begin{equation}
		(\Law (X^n,\bar X^n,W^n,\bar W^n))_{n \in \N} = (\Law(X^{b^n, \sigma^n}, \bar{X}^{\bar{b}^n, \bar{\sigma}^n}, W, \bar W))_{n \in \N},	
	\end{equation}
	$(X^n,\bar X^n,W^n,\bar W^n)\to(Y,\bar Y,W^{\infty},\bar W^\infty)$ pointwise, and $\Law(Y, \bar Y) = \eta$. Moreover, for each $n \in \N$, there exist deterministic maps $F^n, \bar{F}^n$ such that $X^{b^n, \sigma^n} = F^n(W)$ and $\bar{X}^{\bar{b}^n, \bar{\sigma}^n} = \bar{F}^n(\bar{W})$, and so $X^n = F^n(W^n)$ and $\bar{X}^n = \bar{F}^n(\bar{W}^n)$. Therefore $(X^n, \bar{X}^n)$ is a strong solution of the system \eqref{eq:sde-path-dep}, with coefficients $b^n, \bar b^n, \sigma^n, \bar \sigma^n$, driven by $(W^n, \bar W^n)$.
By the equality in law, we can also verify that $(W^n,\bar W^n)$ is a $\rho$-correlated Wiener process in its own filtration.

Let $\varepsilon > 0$. By Lusin's theorem applied to the measurable function $\rho$, we can find a closed set $E\subset [0,1]\times \Omega \times \Omega$ with $m_\varepsilon \coloneqq (\D t\times \mathbb P)(\{(t,\omega)\colon (t,W^n(\omega),\bar W^n(\omega))\notin E\})\leq \varepsilon$ and $\rho|_E$ continuous. We remark that $m_\varepsilon$ is independent of $n\in \mathbb N\cup\{\infty\}$ as it only depends on the joint law of $(W^n,\bar W^n)$. By Tietze's theorem there exists a continuous function $\rho^\varepsilon\colon [0,1]\times \Omega \times \Omega\to[-1,1]$, which coincides with $\rho$ on $E$.
For $n \in \N$, the martingale problem associated with the system for $(X^n,\bar X^n)$ reads as follows: for every bounded $f\colon [0,1]\times \R^2\to\R$ which is differentiable in time, twice differentiable in space, and with corresponding bounded and continuous derivatives, it holds that $R_f^n=0$, where
\begin{align}
&R_f^n\coloneqq \mathbb E\Bigl[ f(T,X^n_1,\bar X^n_1) - f(0,x^n_0,\bar x^n_0)-  \\ & \int_0^1\{\partial_t f + b^n\partial_xf+ \bar b^n\partial_{\bar x}f+\frac{1}{2}(\sigma^n)^2\partial_{xx}f +\frac{1}{2}(\bar \sigma^n)^2\partial_{\bar x\bar x}f  +\rho(t,W^n,\bar W^n)\sigma^n \bar \sigma^n  \partial_{  x\bar x} f \}(t,X^n,\bar X^n)\D t  \Bigr ],
\end{align}
and we identify $f(t, Z, \bar Z) \equiv f(t, Z_t, \bar Z_t)$, for any processes $Z, \bar Z$. On the other hand, we may also define 
\begin{align}
R_f^\infty &\coloneqq \mathbb E\Bigl[ f(T,Y_1,\bar Y_1) - f(0,x_0,\bar x_0)-  \\ &  \int_0^1\{\partial_t f + b\partial_xf+ \bar b\partial_{\bar x}f+\frac{1}{2}\sigma^2\partial_{xx}f +\frac{1}{2}\bar \sigma^2\partial_{\bar x\bar x}f  +\rho(t,W^\infty,\bar W^\infty)\sigma \bar \sigma  \partial_{  x\bar x} f \}(t,Y,\bar Y)\D t  \Bigr ],
\end{align}
and so our goal is to show that  $R_f^\infty=0$ for all $f$ in the aforementioned class of functions. To this end, for $n \in \N$, we introduce $R_f^{n,\varepsilon}$ and $R_f^{\infty,\varepsilon}$, defined analogously to $R_f^n$ and $R_f^\infty$  with $\rho$ replaced by\footnote{Crucially, we do not claim that the martingale problem with initial condition $(x^n_0, \bar x^n_0)$ and coefficients $(b^n,\bar b^n,\sigma^n,\bar\sigma^n)$, driven by a $\rho^\varepsilon$-correlated Brownian motion, is well-posed. Instead we view the introduction of $\rho^\varepsilon,R_f^{n,\varepsilon},R_f^{\infty,\varepsilon}$, as a technical means toward the goal of showing that $R_f^\infty=0$.} $\rho^\varepsilon$, but otherwise unchanged. We claim that $\lim_{n\to\infty} R_f^{n,\varepsilon} =  R_f^{\infty,\varepsilon}$. To see this, note first that since $\rho^\varepsilon$ as well as $f$ and its partial derivatives are continuous, and since $b^n,\bar b^n,\sigma^n,\bar\sigma^n$ converge in the sense of \eqref{eq:unifconvassump_app}, the integrand converges $\D t\times \mathbb P$-almost surely. In turn, since $\rho^\varepsilon$ as well as $f$ and its partial derivatives are bounded, and since $b^n,\bar b^n,\sigma^n,\bar\sigma^n$ have uniform linear growth, for $n \in \N$, we can leverage the uniform moment estimates given at the start of the proof to apply dominated convergence and conclude the desired claim. 
Next, note that $|R_f^{n} - R_f^{n,\varepsilon}|\leq C\varepsilon$, for all $n\in\mathbb N\cup\{\infty\}$, with a constant $C$ depending on $f$ but, crucially, independent of $n$ and $\varepsilon$; this follows again by the uniform linear growth assumption and the uniform moment estimates, which extend to $Y$ and $\bar Y$. Thus $$|R^\infty_f|\leq |R^\infty_f -R_f^{\infty,\varepsilon} | + \lim_{n \to \infty}|R_f^{\infty,\varepsilon} - R_f^{n,\varepsilon} | + \lim_{n \to \infty}|R_f^{n,\varepsilon} - R_f^n | \leq 2 C \varepsilon,$$ and we can conclude by uniqueness of solutions to the martingale problem associated to the system for $(X,\bar X)$.
\end{proof}

\begin{remark}
    An anonymous referee has pointed out the following alternative argument for the proof of Proposition \ref{prop:stab_abstract}, which circumvents the Lusin/Tietze step: The proof starts the same (so we employ the same notation), up to and including the application of Skorokhod representation. Then one goes on to prove that $(Y,W^\infty)$ solves the same martingale problem as $(X,W)$. Similarly, $(\bar Y,\bar W^\infty)$ solves the same martingale problem as $(\bar X,\bar W)$. By assumption, we have $ Y=F( W^\infty)$ for some measurable function $F$, and the laws of $(Y,W^\infty)$ and $(X,W)$ coincide. Similarly,   $ \bar Y=\bar F( \bar W^\infty)$ and the laws of $(\bar Y,\bar W^\infty)$ and $(\bar X,\bar W)$ coincide. Since also the laws of $(W,\bar W)$ and $(W^\infty,\bar W^ \infty)$ coincide, we conclude the equality in law of $(X,\bar X)$ and $(Y,\bar Y)$.
\end{remark}

\bibliographystyle{abbrvnat}
\bibliography{joint_biblio.bib}

\end{document}